\documentclass[a4paper,8pt]{article}
\usepackage[english]{babel}
\usepackage{geometry}
\usepackage[T1]{fontenc}
\usepackage[latin1]{inputenc}
\usepackage{amsmath,amssymb,mathrsfs,amsthm}
\usepackage{soul}
\usepackage{epsfig}
\usepackage{hyperref}
\usepackage{graphicx}
\usepackage{xypic}
\usepackage{fancyhdr}
\fancypagestyle{plain}{
\lhead{}
\rhead{}

}

\usepackage{tocloft}

\newtheorem{theorem}{Theorem}[section]
\newtheorem{proposition}[theorem]{Proposition}
\newtheorem{claim}[theorem]{Claim}
\newtheorem{lemma}[theorem]{Lemma}

\newtheorem{Corollaire}[theorem]{Corollary}
\newtheorem{definition}[theorem]{Definition}

\newtheorem{remarque}[theorem]{Remark}

\newcommand{\vc}{\|\cdot\|}
\newcommand{\C}{\mathbb{C}}

\newcommand{\lra}{\longrightarrow}
\newcommand{\al}{\alpha}
\newcommand{\la}{\lambda}

\newcommand{\R}{\mathbb{R}}
\newcommand{\cl}{\mathcal{C}^\infty}
\newcommand{\p}{\mathbb{P}}
\newcommand{\eps}{\varepsilon}

\newcommand{\si}{\sigma}
\newcommand{\h}{\mathcal{H}}

\newcommand{\N}{\mathbb{N}}
\newcommand{\z}{\overline{z}}
\newcommand{\pt}{\partial}
\newcommand{\dif}{\frac{\pt}{\pt z}}
\title{The spectral theory of generalized Laplacians  associated to integrable metrics on compact Riemann surfaces}
\date{}
\author{Mounir Hajli}
\begin{document}

\maketitle

\begin{abstract}
 We extend the spectral theory of generalized Laplacians to integrable metrics on compact Riemann surfaces. As a consequence,
 we attach in a direct way,  a holomorphic analytic torsion to any integrable metrics. We also provide a different approach to
 define the analytic
 torsion. We prove that both approaches agree.
\end{abstract}

\tableofcontents
\section{Introduction}

The main goal of this paper is the extension, in the context of the Arakelov geometry, of the notion of the holomorphic analytic
torsion to a large class of
singular metrics, with applications to the special case of canonical metrics on $\p^1$, viewed as toric variety. As an application,
we provide an explanation for the results and the  computations
in \cite{Mounir1} and \cite{Mounir2}.\\

Let $X$ be a compact Riemann surface, and $(\mathcal{O},h_{\mathcal{O}})$ the trivial line bundle equipped with a constant
metric over $X$. In this
article, we
prove that  the classical spectral theory of Laplacians can be
extended to the class of integrable metrics on $X$. Namely, for any   integrable metric $h_{X,\infty}$ on $X$, we construct a
Laplacian  attached to $((X,h_{X,\infty});(\mathcal{O},h_{\mathcal{O}}))$, and  we denote it by $\Delta_{X,\infty}$. We  show that
$\Delta_{X,\infty}$  has a infinite discrete and positive spectrum, and we prove that  the associated Zeta function
$\zeta_{\Delta_{X,\infty}}$,  has
the same properties as in the  classical setting. In particular, we establish that $\zeta_{\Delta_{X,\infty}}$ admits a holomorphic
continuation at $s=0$. We set $T\bigl((X,\omega_{X,\infty}),(\mathcal{O},h_{\mathcal{O}}) \bigr):=\zeta'_{\Delta_{X,\infty}}(0)$,
and we call it
the holomorphic analytic  torsion associated to $\bigl((X,\omega_{X,\infty}),(\mathcal{O},h_{\mathcal{O}}) \bigr) $. Moreover,
given
$(h_{X,p})_{p\in \N}$ a sequence of smooth hermitian metrics on $X$, converging uniformly in suitable way to $h_{X,\infty}$, we
prove that  the real sequence
$\bigl(T((X,\omega_{X,p}),(\mathcal{O},h_{\mathcal{O}})) \bigr)_{p\in \N}$ converges to $T((X,\omega_{X,\infty}),(\mathcal{O},h_{\mathcal{O}}) )$.

Let us recall the construction of the holomorphic analytic torsion. For sake of simplicity, we restrict ourselves to compact
riemannian surfaces. Let $X$ be a compact riemannian surface equipped with a smooth metric $h_X$,  we denote by $\omega_X$
the corresponding Kähler form, and let $(\mathcal{O},h_{\mathcal{O}})$ be a
trivial line bundle on $X$, endowed with a constant metric. We can equip $A^{(0,0)}(X)$, the space of  smooth functions on $X$,    with a hermitian
 product  § \eqref{rappelLAPCLA}, and then we consider   $\Delta$
   the Laplacian acting on $A^{(0,0)}(X)$ § \eqref{rappelLAPCLA}.
   It is known that $\Delta$ admits a discrete, positive and infinite
spectrum, see for instance \cite[§ 6]{GH}. Let $\la_1\leq \la_2\leq \ldots$ be the
 nonzero eigenvalues  of $\Delta$, counted with their multiplicities
and order in increasing order. Using the
 spectral theory of heat kernels,  see \cite{heat}, one proves  that for any $s\in \C$ such
 that $\mathrm{Re}(s)>1$, the sum:
 \[
 \zeta_{\Delta}(s)=\sum_{n\geq 1}\frac{1}{\la_n^s},
 \]
 converges absolutely and admits a meromorphic
continuation
 to the whole complex space. Moreover, this continuation
 is holomorphic at $s=0$. According to Ray and Singer \cite{RaySinger}, we define the holomorphic analytic
 torsion associated to $((X,h_X),(\mathcal{O},h_{\mathcal{O}}))$ as follows:
 \[
 T((X,h_X),(\mathcal{O},h_{\mathcal{O}}))=\zeta_{\Delta}'(0)\;\footnote{Notice that in \cite{RaySinger}, by definition $ T((X,h_X),(\mathcal{O},h_{\mathcal{O}}))$
 is equal to $\zeta_{\Delta_1}'(0)$,
 where $\Delta_1$ is the Laplacian associated to $((X,h_X);(\mathcal{O},h_{\mathcal{O}}))$ and acting on $A^{(0,1)}(X)$. But from \cite[(6) p.131]{Soulé}, we have $\zeta_{\Delta}'(0)=\zeta_{\Delta_1}'(0)$.}.
 \]
 As  pointed out in \cite{heat},  the smoothness of the metrics is a necessary condition in order to defined the holomorphic
 analytic
 torsion. In this article, we extend the latter theory to compact Riemann surfaces equipped with integrable metrics.\\

Let us review the contents of the article.  Let $X$ be a compact Riemann surface, and $(\mathcal{O},h_{\mathcal{O}})$ the trivial line bundle equipped with a constant metric on $X$. Let  $h_{X,\infty}$ be an  integrable metric on $X$, see definition \eqref{rappelmetint}, we
will construct  a  spectral theory attached to $h_{X,\infty}$ and $h_{\mathcal{O}}$, and which extend the classical theory. In section
\eqref{rappelLAPCLA},  we recall the definition and the construction of
the Laplacian associated to smooth metrics on compact Riemann surfaces.  We extend in
section \eqref{paragrapheLapX},  this notion  to any integrable metric $h_{X,\infty}$ on $X$, namely we construct a singular
Laplacian denoted by $\Delta_{X,\infty}$ defined on $A^{(0,0)}(X)$. Our first theorem (theorem \eqref{laplaceTX}) is stated as follows:
\begin{theorem}
There exists $(h_{X,u})_{u\geq 1}$ a sequence of smooth hermitian metrics on $X$, converging uniformly to $h_{X,\infty}$ such that:
\begin{enumerate}
\item
\[
\underset{u\mapsto \infty}{\lim}\bigl\| \Delta_{{}_{X,u}}\xi \bigr\|^2_{L^2,u}=\bigl\|
\Delta_{{}_{X,\infty}}\xi
\bigr\|^2_{L^2,\infty}<\infty,
\]
\item $\Delta_{X,\infty}$ is a linear operator from $A^{(0,0)}(X)$ to $\h_0(X)$ (where
 $\h_0(X)$ is the completion of $A^{(0,0)}(X)$ with respect to the $L^2$-norm).
\item
\[
\bigl(\Delta_{{}_{X,\infty}}\xi,\xi' \bigr)_{L^2,\infty}=
\bigl(\xi,\Delta_{{}_{X,\infty}}\xi' \bigr)_{L^2,\infty},
\]

\item
\[
\bigl(\Delta_{{}_{X,\infty}}\xi,\xi \bigr)_{L^2,\infty}\geq 0,
\]
\end{enumerate}

for any $\xi,\xi'\in A^{0,0}(X)$.

\end{theorem}
Since  $h_{X,u}$ is smooth, a classical fact, see \cite[p.94]{GH}, asserts that $(I+\Delta_{X,u})^{-1}$  is a compact operator on $\h_0(X)$, for any $u\geq 1$.
 It is natural to ask if $I+\Delta_{X,\infty}$ is invertible in a suitable space. The answer to this question is given in the
 following  theorem
(theorem \eqref{invertibleoperator}):
 \begin{theorem} The operator $\Delta_{X,\infty}$ admits a maximal selfadjoint extension to a space, denoted by $\h_2(X)$. We
 denote this extension
also by $\Delta_{X,\infty}$. The operator $I+\Delta_{X,\infty}$ is invertible, more precisely, we have:
\[
(I+\Delta_{X,\infty})(I+\Delta_{X,\infty})^{-1}=I,
\]
on $\h_0(X)$, where $I$ is the operator identity of $\h_0(X)$.
\[
(I+\Delta_{X,\infty})^{-1}(I+\Delta_{X,\infty})=I,
\]
on $\h_2(X)$, where   $I$ is the operator identity of $\h_2(X)$.
\end{theorem}

The proof of this theorem will be split into two steps, but before   we  need to establish two
 technical results \eqref{bornelapbelt} and \eqref{proposition1}. In the first step, we  prove
 the following theorem
 (theorem
 \eqref{deltacompactoperateur}):
\begin{theorem}
The sequence $\bigl((I+\Delta_{X,u})^{-1} \bigr)_{u\geq 1}$ converges to a compact operator denoted by
$(I+\Delta_{X,\infty})^{-1}:\h_0(X)\lra \h_0(X)$, with respect to $L^2_{X,\infty}$-norm.
\end{theorem}
We  extend the operator $\Delta_{X,\infty}$ in the second step, this is the goal of the
subsection \eqref{Freidrichs},
where we review the notion of selfadjoint extension of Laplacians in the classical setting. Using
theorem \eqref{deltacompactoperateur}, we establish that $\Delta_{X,\infty}$ admits a maximal positive selfadjoint extension. \\

Next, we prove that $\Delta_{X,\infty}$ has a infinite discrete and positive spectrum, and by an operator theory argument it admits
a heat
kernel $e^{-t\Delta_{X,\infty}}$ for any $t>0$. The following  theorem
(see \eqref{variationEu}) shows that $e^{-t\Delta_{X,\infty}}$ is a limit of a sequence $\bigl(e^{-t\Delta_{{}_{X,u}}}\bigr)_u$
of heat kernels attached to $(h_{X,u})_{u\geq 1}$, a sequence of smooth metrics on $X$.
 \begin{theorem}
 For any $t>0$, we have:
\[
 \bigl(e^{-t\Delta_{{}_{X,u}}}\bigr)_u\xrightarrow[u\to+\infty]{}e^{-t\Delta_{X,\infty}},
\]
In particular, $e^{-t\Delta_{{}_{X,\infty}}}$ is a compact operator from $\h_0(X)$ to $\h_0(X)$.
\end{theorem}
In order to prove this theorem, we show first that the sequence $\bigl(e^{-t\Delta_{{}_{X,u}}}\bigr)_u$ converges to
a limit, for any $t>0$. Then, we conclude using the uniqueness of the heat kernel.\\

In subsection \eqref{tracezetainfty}, we study the spectral properties of $\Delta_{X,\infty}$. We introduce $\theta_{X,\infty}$,
the associated
Theta function
  and we prove, as in the classical theory, that $\theta_{X,\infty}(t)$ is finite for any $t>0$, that is the
theorem  \eqref{thetafinite}. The  theorem  \eqref{key2} is the core of the article: If we denote by $(\la_{\infty,k})_{k\in \N}$
the sequence of the eigenvalues of $\Delta_{X,\infty}$  counted with their multiplicities,
and ordered in increasing order, we have:
\begin{theorem}
We have, for any $t>0$ fixed:
\[
\bigl(\theta_{X,u}(t)\bigr)_{u\geq 1}\xrightarrow[u\mapsto \infty]{} \theta_{X,\infty}(t),
\]
and
\[
\zeta_{X,\infty}(s):=\frac{1}{\Gamma(s)}\int_0^\infty t^{s-1}\theta_{X,\infty}(t)dt=\sum_{k=1}^\infty \frac{1}{\la_{\infty,k}^s},
\]
is finite for any $s\in \C$, such that $\mathrm{Re}(s)>1$. This
function of $s$ admits a meromorphic continuation to the
whole complex plane with a pole at $s=1$, and it is holomorphic at $s=0$. We have,
\[
\zeta'_{X,\infty}(0)=\int_1^\infty \frac{\theta_{X,\infty}(t)}{t}dt+\gamma b_{\infty,-1}-b_{\infty,0}+\int_0^1 \frac{\rho_{X,\infty}(t)}{t}dt=\lim_{u\mapsto \infty} \bigl(\zeta'_{X,u}(0)\bigr)_{u\geq 1},
\]
where $b_{\infty,-1}$, $b_{\infty,0}$ are real numbers and $\rho_{X,\infty}$ is a real function  such that $\theta_{X,\infty}(t)=\frac{b_{\infty,-1}}{t}+b_{\infty,0}
+ \rho_{X,\infty}(t)$, and  $\rho_{X,\infty}(t)=O(t)$ for $t>0$ sufficiently small.
\end{theorem}
The proof of this theorem relies on some technical lemmas, and on a critical result which gives a uniform lower bound for the
first nonzero eigenvalue of $\Delta_{X,u}$ for any $u\geq 1$.\\

We provide in theorem \eqref{convergenceAnomaly} a new approach to extend the Quillen metrics to integrable metrics on compact
riemannian surface; we prove the following result:
\begin{theorem}
We keep the same assumptions. For any $p\in \N$, let $ h_{Q,((X,\omega_{X,p});(\mathcal{O},h_{\mathcal{O}})
)}$ be the Quillen metric associated to $\bigl((X,\omega_{X,p});(\mathcal{O},h_{\mathcal{O}}) \bigr)$. We have, the sequence $\Bigl(
h_{Q,((X,\omega_{X,p});(\mathcal{O},h_{\mathcal{O}}) )}\Bigr)_{p\in \N}$ converges to a limit, which does not depend on the choice of
$\bigl(h_{X,p}\bigr)_{p\in \N}$.  We denote this limit by $h_{Q,((X,\omega_{X,\infty});(\mathcal{O},h_{\mathcal{O}}) )}$.
\end{theorem}

We finish this section, by   comparing  both  methods. This is done in theorem \eqref{compare2methods}, it is stated as follows :
\begin{theorem}
We have,

\[
h_{Q,((X,\omega_{X,\infty});(\mathcal{O},h_{\mathcal{O}}) )}=h_{L^2,((X,\omega_{X,\infty});(\mathcal{O},h_{\mathcal{O}}) )}\exp\bigl(\zeta_{X,\infty}'(0)\bigr).
\]
\end{theorem}

We recall and review in section \eqref{rappelclassic} some classical notions used through this article. In section
\eqref{Quelqueslemmes}, we prove some technical results. \\

\noindent\textbf{Acknowledgments}: This work is a part of the author's
Ph.D thesis \cite{these} under the supervision of Vincent Maillot. I thank him for his suggestions and encouragement throughout
my work on this project. I  thank Gerard Freixas for the many hours of mathematical discussions on the topic of this work, and
also Dennis Eriksson.

\section{The Laplacian on compact Riemann Surface }\label{rappelLAPCLA}
Let us recall the  construction of the generalized   Laplacian $\Delta$ acting on $A^{(0,0)}(X)$. We will emphasizes that this construction does not require the smoothness of $h_X$, we can assume
 that $h_X$  is only  continuous.

Let $h_X$ be  a continuous hermitian metric on $TX$, and  $h_{\mathcal{O}}$ a constant metric on $\mathcal{O}$. We denote   by $\omega_X$ the normalized Kähler form associated to $h_X$, given on any local chart of $X$ as follows:
\[
 \omega_X=\frac{i}{2\pi}h_X\Bigl(\frac{\partial}{\partial z},\frac{\partial}{\partial z}\Bigr)dz\wedge d\z.
\]

This metric induces a  metric on the space of differential forms of type  $(0,1)$. Tensoring by $h_{\mathcal{O}}$, the metric of $\mathcal{O}$, we obtain
a  pointwise inner  product at  any $x\in X$: $(s(x),t(x))$ for two section  of $A^{0,q}(X)=A^{0,q}(X)\otimes _{\cl(X)}A^0(X)$,
and $q=0 $ or $1$. The $L^2$ inner product of two sections $s,t\in A^{0,q}(X)$ is given by the formula:
{{}
\[
 (s,t)_{L^2}=\int_X \bigl(s(x),t(x)\bigr)\omega_X.
\]}

The  Cauchy-Riemann operator $\overline{\partial}_{\mathcal{O}}$ acts on  the  forms of type $(0,q)$ with values in  $\mathcal{O}$. We have the  Dolbeault
complex:
{{}
\[
0\lra A^{0,0}(X)\overset{\overline{\partial}_{\mathcal{O}}}{\longrightarrow}A^{0,1}(X)\lra 0 
\]
}
Its cohomology is known to be  the  sheaf cohomology of  $X$ with  coefficients in $\mathcal{O}$, cf. for example \cite{GH}.\\

The operator $\overline{\partial}_{\mathcal{O}}$ admit a formal adjoint for the inner  product $L^2$; namely an application
{{}
\[
 \overline{\partial}_{\mathcal{O}}^\ast:A^{0,1}(X)\lra A^{0,0}(X)
\]}
which verify
\[
 \bigl(s,\overline{\partial}_{\mathcal{O}}^\ast t\bigr)_{L^2}=\bigl(\overline{\partial}_{\mathcal{O}}s,t\bigr)_{L^2}.
\]
for any $s\in A^{0,q}(X)$ et $t\in A^{0,q+1}(X)$.
It follows from the definition that the operator $\overline{\partial}_{\mathcal{O}}^\ast$ is given by the
formula:

\[
\overline{\partial}_{\mathcal{O}}^\ast=-\ast_{0}^{-1}\overline{\partial}_{K_X\otimes \mathcal{O}^\ast}\ast_{1},
\]
see for instance \cite[§.5]{Voisin}, where $\ast_0$ and $\ast_1$ are the following applications:
\[
 \ast_{0}:A^{0,0}(X)\lra A^{1,1}(X,\mathcal{O}^\ast),\]
and,
\[
\ast_{1}: A^{0,1}(X)\lra A^{1,0}(X,\mathcal{O}^\ast).\\
\]
They are the unique  applications which satisfy the following:
\[
 f(x)\wedge \ast_{0} (g(x))=f(x)\overline{g(x)}\omega_x,
\]
and
\[
\bigl(f d\overline{z}\bigr) \wedge\ast_{1}(gd\overline{z}) =\bigl(f\,d\overline{z}(x),g\,d\overline{z}(x)\bigr)_x \omega_X(x),
\]
for any $x\in X$  such that $f,g\in A^{0,0}(X)$. Notice that in order to  define $\ast_{0}$ and $\ast_{1}$, we do not need that
$h_X$ to be smooth. We can show easily that these morphisms can be written   respectively on a local chart, as follows:

\[
 \ast_{0} (g)=\overline{g}\,\omega_X.
\]
and,
\begin{equation}\label{starhodge}
\ast_{1}(gd\overline{z})=-\overline{g} dz.
\end{equation}

We denote by $\Delta_{\overline{\mathcal{O}}}^0$, or    by $\Delta_{\overline{\mathcal{O}}}$, or simply by $\Delta$  the operator $\overline{\partial}_{\mathcal{O}}^\ast
\overline{\partial}_{\mathcal{O}}$ on $A^{0,0}(X)$.  Following  \cite[Definition 8.1, p.101]{Soulé}, we call it the generalized  Laplacian
associated to $h_X$ and $h_{\mathcal{O}}$.

\begin{remarque}
\rm{Even $h_X$ is not smooth, the operator $\Delta_{\overline{\mathcal{O}}}=\overline{\partial}_{\mathcal{O}}^\ast
\overline{\partial}_{\mathcal{O}}$ is well defined.}
\end{remarque}
\begin{lemma}
Let $(X,h_X)$ be a compact Riemann surface such that  $h_X$ is continuous, and $(\mathcal{O},h_{\mathcal{O}})$ the trivial line bunlde equipped with a constant metric such that $h_{\mathcal{O}}$ is
smooth.
The Laplacian  $\Delta $ associated to $h_X$ and $h_{\mathcal{O}}$, is given locally as follows:
\begin{equation}\label{explap11}
 \Delta(f)
=-h_X\Bigl(\frac{\partial}{\partial z},\frac{\partial}{\partial z}\Bigr)^{-1}\frac{\partial^2 f}{\partial z\pt \z},
\end{equation}
for any $ f\in A^{0,0}(X)$,
where $\{\dif\}$ is a local holomorphic basis of $TX$.


\end{lemma}

\begin{proof}
See for instance \cite[definition 2.3.3]{compactriemannsurfaces}.
\end{proof}
\section{The variation of metric on $TX$ and the operator $\Delta_{X,\infty}$}\label{paragrapheLapX}
In this section, we construct a singular Laplacian, $\Delta_{X,\infty}$ attached to an integrable metric
on $X$, and  we study its spectral properties.
\subsection{The singular Laplacian $\Delta_{X,\infty}$}
Let $X$ be a compact riemannian surface, and we equip $\mathcal{O}$ the trivial line bundle on $X$ with constant metric $h_\mathcal{O}$ such that $h_\mathcal{O}(1,1)=1$. We endow $X$ with an integrable metric
$h_{X,\infty}$. By definition, there exist $h_{1,\infty}$ and $h_{2,\infty}$ two admissible metrics (see
\eqref{admissibledefinition}) such that
$h_{X,\infty}=h_{1,\infty}\otimes h_{2,\infty}^{-1}$. Let $(h_{1,n})_{n\in \N}$ and  $(h_{2,n})_{n\in \N}$   be two sequences
of smooth semipositive metrics which converge uniformaly to $h_{1,\infty}$ and $h_{2,\infty}$, respectively. Let
$h_{X,n}:=h_{n,1}\otimes h_{2,n}^{-1}$ for any $n\in \N$, and we consider
 the family $\bigl(h_{X,u}\bigr)_{u>1}$ attached to this
sequence as  in \eqref{suitefamille}. Recall that  $h_{X,u}$ is a smooth hermitian metric on $TX$. We denote by  $\omega_{X,u}$ the
normalized volume form attached and by $\Delta_{{}_{X,u}}$ the Laplacian attached to $h_{X,u}$ and  $h_{\mathcal{O}}$ for any $u\in
]1,\infty[$. \\

For all $u\in ]1,\infty]$, we denote by $L^2_{X,u}$ (resp. $(\cdot,\cdot)_{L^2,u}$)  the hermitian norm  (resp.   the hermitian
product) induced by $h_{X,u}$ and
$h_{\mathcal{O}}$ on $A^{(0,0)}(X)$, as in the previous section. We denote by $\h_0(X,u)$ the completion of $A^{(0,0)}(X)$ with respect
to $L^2_{X,u}$-norm.
\begin{lemma}
The family $(L^2_{X,u}-norms)_{u\geq 1}$ forms a sequence of uniformly equivalent norms on $A^{(0,0)}(X)$. In particular,
$\h_0(X,u)$ does not depend on $u$, we will denote it by $\h_0(X)$.
\end{lemma}
\begin{proof}
It suffices to notice that $(h_{X,u})_{u\geq 1}$ forms a bounded sequence, and we conclude using the compactness of $X$.
\end{proof}

\begin{definition}\rm{
\rm{For all $\xi\in A^{0,0}(X)$, we set:
\[\Delta_{{}_{X,\infty}}\xi:=-{}h_{X,\infty}\Bigl(\dif,\dif\Bigr)^{-1}\dif\Bigl(
\frac{\pt \xi}{\pt \z} \Bigr).
\]
where  $\bigl\{\dif \bigr\}$ is a local holomorphic basis of $TX$.
We say,  $\Delta_{X,\infty}$ is the Laplacian  attached to $h_{X,\infty}$ and $h_{\mathcal{O}}$.
}}
\end{definition}

In the following theorem, we prove that  $\Delta_{X,\infty}$ is a  linear operator  defined over  $A^{(0,0)}(X)$ with values
belong to  $\h_0(X)$.
 \begin{theorem}\label{laplaceTX}
Keeping the same hypothesis as before, we have:
\begin{enumerate}
\item
\[
\underset{u\mapsto \infty}{\lim}\bigl\| \Delta_{{}_{X,u}}\xi \bigr\|^2_{L^2,u}=\bigl\|
\Delta_{{}_{X,\infty}}\xi
\bigr\|^2_{L^2,\infty}<\infty,
\]
\item $\Delta_{X,\infty}$ is a linear operator from $A^{(0,0)}(X)$ to $\h_0(X)$.
\item
\[
\bigl(\Delta_{{}_{X,\infty}}\xi,\xi' \bigr)_{L^2,\infty}=
\bigl(\xi,\Delta_{{}_{X,\infty}}\xi' \bigr)_{L^2,\infty},
\]

\item
\[
\bigl(\Delta_{{}_{X,\infty}}\xi,\xi \bigr)_{L^2,\infty}\geq 0,
\]
\end{enumerate}

for any $\xi,\xi'\in A^{0,0}(X)$.

\end{theorem}

\begin{proof} Let  $\xi\in A^{(0,0)}(X)$.  We have for any $u>1$:
\[
\bigl\| \Delta_{{}_{X,u}}\xi \bigr\|^2_{L^2,u}=\int_{x\in X}\bigl(\Delta_{X,u}\xi,\Delta_{X,u}\xi\bigr)_x
\omega_{X,u}=\frac{i}{2\pi}\int_{x\in X} \Bigl(\Delta_{{}_{X,u}}\xi, \Delta_{{}_{X,u}}\xi \Bigr)_x\, h_{X,u
}\Bigl(\dif(x),\dif(x)\Bigr)_x dz_x\wedge d\z_x\\,
\]
where  $\{\dif(x)\}$ is local basis  $TX$ in a open subset $U$ containing  $x$. \\

 For any $u\geq 1$, the  Laplacian $\Delta_{X,u}$ has the following expression:
$x\in U$ :
\[
\Delta_{{}_{X,u}}\xi=-h\Bigl(\dif(x),\dif(x)\Bigr)^{-1}\dif_x\Bigl(\frac{\pt \xi}{\pt \z_x} \Bigr).
\]
 We have:
{
\begin{align*}
 0\leq h_{X,u }&\Bigl(\dif(x),\dif(x)\Bigr)_x\Bigl(\Delta_{{}_{X,u}}\xi, \Delta_{{}_{X,u}}\xi \Bigr)_x\\
 &= h_{X,u }\Bigl(\dif,\dif\Bigr)\biggl( {}h_{X,u}\Bigl(\dif,\dif\Bigr)^{-1}{}\dif\Bigl({}\frac{\pt \xi}{\pt \z_x} \Bigr){},  {}h_{X,u}\Bigl(\dif,\dif\Bigr)^{-1}{}\dif\Bigl({}\frac{\pt \xi}{\pt \z_x} \Bigr){}\biggr)  \\
 &=h_{X,u }\Bigl(\dif,\dif\Bigr)^{-1}_x{}{}\dif\Bigl({}\frac{\pt \xi}{\pt \z} \Bigr)\frac{\pt}{\pt \z}\Bigl({}\frac{\pt \overline{\xi}}{\pt z}
 \Bigr)\\
&\leq \frac{h_{X,\infty }\Bigl(\dif,\dif\Bigr)_x}{h_{X,u }\Bigl(\dif,\dif\Bigr)_x} h_{X,\infty }\Bigl(\dif,\dif\Bigr)_x\Bigl(\Delta_{{}_{X,\infty}}\xi, \Delta_{{}_{X,\infty}}\xi \Bigr)_x,
\end{align*}}
Note that  $x\mapsto h_{X,\infty }\Bigl(\dif,\dif\Bigr)_x h_{X,u }\Bigl(\dif,\dif\Bigr)_x^{-1} $ is the restriction on $U$ of
 a global continuous function  on $X$, hence bounded. Recall that  $\bigl(h_{X,u}\bigr)_u\xrightarrow[u\mapsto\infty]{} h_{X,\infty}$. Using
a partition of unity and  according to the dominated convergence theorem, we get:
{\allowdisplaybreaks
\begin{align*}
\bigl\| \Delta_{{}_{X,\infty}}\xi &\bigr\|^2_{L^2,\infty}=\frac{i}{2\pi}\int_X h_{X,\infty
}\Bigl(\dif,\dif\Bigr)_x\Bigl(\Delta_{{}_{X,\infty}}\xi, \Delta_{{}_{X,\infty}}\xi \Bigr)_x\,dz\wedge d\z\\
&=\frac{i}{2\pi}\int_X h_{X,\infty}\Bigl(\dif,\dif\Bigr)^{-1}{}{}{}
\dif\Bigl({}\frac{\pt \xi}{\pt \z} \Bigr)\frac{\pt}{\pt \z}\Bigl({}\frac{\pt \overline{\xi}}{\pt z}
\Bigr){} dz\wedge d\z\\
&=\frac{i}{2\pi }\int_X\underset{u\mapsto
\infty}{\lim}h_{X,u}\Bigl(\dif,\dif\Bigr)^{-1}{}{}{}\dif\Bigl({}\frac{\pt \xi}{\pt
\z} \Bigr)\frac{\pt}{\pt \z}\Bigl({}\frac{\pt \overline{\xi}}{\pt z} \Bigr){}dz\wedge d\z\\
&=\frac{i}{2\pi }\underset{u\mapsto \infty}{\lim}\int_Xh_{X,u
}\Bigl(\dif,\dif\Bigr)^{-1}{}{}{}\dif\Bigl({}\frac{\pt \xi}{\pt \z}
\Bigr)\frac{\pt}{\pt \z}\Bigl({}\frac{\pt \overline{\xi}}{\pt z} \Bigr){} dz\wedge d\z\\
&=\underset{u\mapsto \infty}{\lim}\bigl\| \Delta_{{}_{X,u}}\xi \bigr\|^2_{L^2,u},
\end{align*}}
this proves the first claim.\\

The linearity in the second claim is obvious, and the second part follows from the first one.

Now let $\xi, \xi'\in A^{(0,0)}(X)$, we have
\[
\begin{split}
\bigl(\Delta_{{}_{X,\infty}}(\xi),\xi' \bigr)_{L^2,\infty}&=\int_X {}\dif\Bigl(\frac{\pt \xi}{\pt \z} \Bigr)\,\overline{\xi'}\,dz\wedge d\z\\
&=\int_X \frac{\pt \xi}{\pt \z}\frac{\pt \overline{\xi'}}{\pt z}dz\wedge d\z, \quad \text{by Stokes' theorem}\\
&=\bigl(\xi,\Delta_{{}_{X,\infty}}(\xi' )\bigr)_{L^2,\infty}.
\end{split}
\]
We infer,
{\allowdisplaybreaks
\begin{align*}
\bigl(\Delta_{{}_{X,\infty}}\xi,\xi \bigr)_{L^2,\infty}&=\frac{i}{2\pi}\int_X \frac{\pt \xi}{\pt \z}\frac{\pt \overline{\xi}}{\pt z}dz\wedge d\z\geq 0.
\end{align*}}
\end{proof}

We prove now, some technical results which will allow us to study the spectral properties of $\Delta_{X,\infty}$. We
introduce the following function:
 \[\delta_X(u):=\sup_{x\in X}\Bigl|\frac{\pt }{\pt u}\Bigl(\log h_{X,u}(\dif,\dif)^{-1} \Bigr)(x) \Bigr|\quad \forall u>1,\]
where $\{\dif\} $ is  a local holomorphic basis of  $TX$.\\

 Note that  $\delta_X$ does not depend on the choice of the basis. Indeed, since
$h_{X,u}(\dif,\dif)=(1-\rho(u))h_{X,p-1}(\dif,\dif)+\rho(u)h_{X,p}(\dif,\dif)$, for any $p\in \N^\ast$, $ u\in [p-1,p]$ and $\{\dif \}$ a local
holomorphic basis of $TX$.  Then \begin{align*}
\frac{\pt }{\pt u}\log h_{X,u}(\dif,\dif)^{-1}
&=\rho(u)\frac{h_{X,p-1}(\dif,\dif)-h_{X,p}(\dif,\dif)}{h_{X,u}(\dif,\dif)}\\
&=\rho(u)\frac{h_{X,p}(\dif,\dif)}{h_{X,u}(\dif,\dif)}\frac{h_{X,p-1}(\dif,\dif)-h_{X,p}(\dif,\dif)}{h_{X,p}(\dif,\dif)}
\end{align*}
Which is clearly a  well defined continuous function on $X$.  We know that
$(h_{X,p})_{p\in \N}$ converges uniformly to  $h_{X,\infty}$, then there is a constant $c_1$ such that:
\begin{equation}\label{deltaXU}
\delta_X(u)\leq c_1\biggl| \frac{h_{X,[u]}-h_{X,[u]+1}}{h_{X,[u]+1}}\biggr| \quad \forall\, u\geq 1,
\end{equation}

where $[u]$ is the round down of $u$.

\begin{proposition}\label{bornelapbelt}
 We have the following:
\[
 \biggl\| \frac{\pt \Delta_{{}_{X,u}}}{\pt u}\xi \biggr\|_{L^2,u}\leq \delta_X(u)\bigl\|\Delta_{{}_{X,u}}\xi\bigr\|_{L^2,u},
\]
for any $\xi\in A^{0,0}(X)$ and $u>1$.
\end{proposition}
\begin{proof}
Fix $\xi\in A^{0,0}(X)$. Using the expression of the Laplacian, we have:
\[
 \frac{\pt \Delta_{{}_{X,u}}}{\pt u}\xi=\frac{\pt }{\pt u}\bigl(\log h_{X,u}(\dif,\dif)^{-1} \bigr)\Delta_{{}_{X,u}}\xi \quad \forall \, \xi\in A^{(0,0)}(X).
\]
Recall that
   $\frac{\pt }{\pt u}\bigl(\log h_{X,u}(\dif,\dif)^{-1} \bigr)$ is a continuous function on  $X$ which does not depend on the choice
   of the basis.\\

We have,
\begin{equation}\label{convcomp}
\begin{split}
\biggl( \frac{\pt \Delta_{{}_{X,u}}}{\pt u}\xi, \frac{\pt \Delta_{{}_{X,u}}}{\pt u}\xi\biggr)_{L^2,u}&=\int_X \biggl|\frac{\pt }{\pt u}\bigl(\log h_{X,u}(\dif,\dif)^{-1} \bigr) \biggr|^2h_{X,u}\bigl(\Delta_{{}_{X,u}}\xi,\Delta_{{}_{X,u}}\xi\bigr)\omega_{X,u}\\
&\leq |\delta_X(u)|^2\int_X h_{X,u}(\Delta_{{}_{X,u}}\xi,\Delta_{{}_{X,u}}\xi)\omega_{X,u}\\
& =|\delta_X(u)|^2 \bigl( \Delta_{{}_{X,u}}\xi,  \Delta_{{}_{X,u}}\xi\bigr)_{L^2,u}.
\end{split}
\end{equation}
which yields the assertion.
\end{proof}

Recall that for any $1<u<\infty$, there exists a compact operator $(I+\Delta_{X,u})^{-1}$ on $\h_0(X)$ such that:
\[
(I+\Delta_{X,u})^{-1}(I+\Delta_{X,u})=I_0,
\]
where $I_0$ means the identity operator of $A^{(0,0)}(X)$. We study next the variation of $(I+\Delta_{X,u})^{-1}$ with respect
to $u$, but first let us recall  the following fact: Let  $\Delta$ be a Laplacian associated to smooth metrics. Then
\begin{equation}\label{normeleq1}
\|(\Delta+I)^{-1}\|\leq 1,
\end{equation}
  where $\|\cdot\|$ is the induced metric. Indeed,  we know that the eigenvectors of $\Delta$ form
 a complete orthonormal system for the completion of $A^{0,0}(X)$ with respect to the metrics. So, if we denote  $(\phi_i)_{i}$
 an  orthonormal basis of  eigenvectors of $\Delta$, then for any $\xi \in \overline{A^{0,0}(X)}$, there exists $(a_j)_{j\in \N}$
 a sequence of complex numbers such
 that $  \xi =\sum_i a_i\phi_i $ and one  checks that
\[
\bigl\|(\Delta+I)^{-1}\xi\bigr\|^2=\Bigl\| \sum_i \frac{a_i}{\la_i+1}\phi_i \Bigr\|^2\leq \sum_i |a_i|^2\|\phi_i\|^2=\|\xi\|^2.
\]

\begin{proposition}\label{proposition1}
We have,
\[
\biggl\|\frac{\pt }{\pt u}(I+\Delta_{{}_{X,u}})^{-1}\biggr\|_{L^2,\infty}\leq c_3|\delta_X(u)| \quad \forall\, u>1.
\]
where $c_3$ is a constant.
\end{proposition}

\begin{proof}
We have,
\[
 \frac{\pt }{\pt u}\bigl(I+\Delta_{{}_{X,u}}\bigr)^{-1}=-\bigl(I+\Delta_{{}_{X,u}}\bigr)^{-1}\frac{\pt \Delta_{{}_{X,u}}}{\pt u} \bigl(I+\Delta_{{}_{X,u}}\bigr)^{-1}\quad \forall \, u>1.
\]

Now let $\eta \in A^{(0,0)}(X)$ and put  $\xi=(\Delta_{{}_{X,u}}+I)^{-1}\eta$, we have
\[
\begin{split}
\biggl\| \frac{\pt \Delta_{X,u}}{\pt u} \, (\Delta_{{}_{X,u}}+I)^{-1}\eta\biggr\|^2_{L^2,u}&=\biggl\| \frac{\pt \Delta_{{}_{X,u}}}{\pt u} \, \xi\biggr\|^2_{L^2,u}\\
&\leq \bigl|\delta_X(u)\bigr|^2 \bigl(\Delta_{{}_{X,u}}\xi,\Delta_{{}_{X,u}}\xi\bigr)_{L^2,u}\quad \text{by} \;\eqref{bornelapbelt}\\
&= |\delta_X(u)|^2 \bigl(\Delta_{{}_{X,u}}(\Delta_{{}_{X,u}}+I)^{-1}\eta,\Delta_{{}_{X,u}}(\Delta_{{}_{X,u}}+I)^{-1}\eta\bigr)_{L^2,u}\\
&=|\delta_X(u)|^2 \Bigl(\eta- (\Delta_{{}_{X,u}}+I)^{-1}\eta,(\Delta_{{}_{X,u}}+I)^{-1}\eta  \Bigr)_{L^2,u}\\
&\leq 2|\delta_X(u)|^2 \|\eta\|_{L^2,u}^2,
\end{split}
\]
Thus
\[
\biggl\| \frac{\pt}{\pt u} \Delta_{{}_{X,u}}\cdot (\Delta_u+I)^{-1}\Bigr\|^2_{L^2,u}\leq 2\big|\delta_X(u)\bigr|^2\quad \forall \, u>1.
\]

Since  $\|\cdot\|_{L^2,\infty}$ and   $ \|\cdot\|_{L^2,u}$ for any $u>1$  are uniformly equivalent, there exists $c_2$
a constant such that:
\begin{equation}\label{ineqinfty1}
\biggl\| \frac{\pt}{\pt u} \Delta_{{}_{X,u}}\cdot (\Delta_u+I)^{-1}\Bigr\|^2_{L^2,\infty}\leq c_2\big|\delta_X(u)\bigr|^2\quad \forall \, u>1.
\end{equation}

By the same argument and using the first claim, we can find a constant $c_3$ such that
\begin{align*}
\biggl\|\frac{\pt }{\pt u}\bigl(I+\Delta_{{}_{X,u}}\bigr)^{-1}\biggr\|_{L^2,\infty}&=\biggl\| \bigl(I+\Delta_{{}_{X,u}}\bigr)^{-1}\frac{\pt \Delta_{{}_{X,u}}}{\pt u} \bigl(I+\Delta_{{}_{X,u}}\bigr)^{-1} \biggr\|_{L^2,\infty} \\
&\leq \bigl\|(\Delta_{{}_{X,u}}+I)^{-1} \bigr\|_{L^2,\infty}\biggl\| \frac{\pt}{\pt u} \Delta_{{}_{X,u}}\cdot (\Delta_{{}_{X,u}}+I)^{-1}\biggr\|_{L^2,\infty}\\
&\leq c_3|\delta_X(u)|\quad \forall \,u>1.
\end{align*}
That is,
\[
\biggl\|\frac{\pt }{\pt u}(I+\Delta_{{}_{X,u}})^{-1}\biggr\|_{L^2,\infty}\leq c_3|\delta_X(u)| \quad \forall\, u>1.
\]

\end{proof}

\begin{theorem}\label{deltacompactoperateur}
The sequence $\bigl((I+\Delta_{X,u})^{-1} \bigr)_{u\geq 1}$ converges to a compact operator denoted by
$(I+\Delta_{X,\infty})^{-1}:\h_0(X)\lra \h_0(X)$, with respect to $L^2_{X,\infty}$-norm.
\end{theorem}
\begin{proof}
According to \eqref{deltaXU}, we may assume that $\bigl|\delta_X(u)\bigr|=O\bigl(\frac{1}{u^2}\bigr)$. This yields for any  $q>p$
\begin{align*}
\Bigl\|\bigl(\Delta_{{}_{X,p}}+I\bigr)^{-1}-\bigl(\Delta_{{}_{X,q}}+I\bigr)^{-1}\Bigr\|_{L^2,\infty}&=\biggl\|\int_p^q\frac{\pt }{\pt u}\bigl(I+\Delta_{{}_{X,u}}\bigr)^{-1}du\biggr\|_{L^2,\infty}\\
&\leq \int_p^qc_3 |\delta_X(u)|du\quad \text{by}\;\eqref{proposition1}\\
&=\int_p^q O(\frac{1}{u^2})du\\
 &=O\Bigl(\frac{1}{q}-\frac{1}{p}\Bigr)\quad \forall\, p,q\gg1.
\end{align*}

Therefore, the sequence of compact operators $\bigl((\Delta_{{}_{X,p}}+I)^{-1}\bigr)_{p\in \N}  $ converges to a  operator
which we denote by $(I+\Delta_{X,\infty})^{-1}$. This operator is compact according to \eqref{operateurcompact}.
\end{proof}
\subsection{A maximal  positive selfadjoint extension of $\Delta_{X,\infty}$}\label{Freidrichs}
The main goal of this
paragraph is to prove that $\Delta_{X,\infty}$ admits a maximal selfadjoint extension to
a large subspace denoted by $\h_2(X)$.

Let us begin first by reviewing some facts about the notion of selfadjoint extension of Laplacians. For simplicity, we
restrict our selves to  Compact Riemann surfaces. Let $X$ be a compact Riemann surface. Let $\omega_X$ be normalized Kähler form on $X$, and $h_{\mathcal{O}}$ a   constant metric on  $\mathcal{O}$. For any $\phi, \psi\in A^{(0,0)}(X)$, we
define an hermitian product $(\phi,\psi)$ as before,
 the corresponding norm will be denoted by $\|\cdot\|$, and  we will call it the $L^2$-norm. We let
$\h_0(X)$ be the completion of the pre-Hilbert space $\bigl(A^{(0,0)}(X),(,)\bigr)$. One proves $\h_0(X)$ does not depend on
the choice
of the metrics. More precisely,  given two continuous metrics on $X$  and using the compactness of $X$, we obtain  two
equivalent metrics on $A^{(0,0)}(X)$.\\

When the metric of $X$ is smooth, it is known that the Laplacian has a complete orthonormal sequence
of smooth $\mathcal{O}$-valued eigenfunctions $\phi_0,\phi_1,\phi_2,\ldots$ in $\h_0(X)$.   In particular, we have:
\[
 \h_0(X)=\bigl\{\phi=\sum_{k=0}^\infty a_k\phi_k\, \bigl|\,  \|\phi\|^2=\sum_{k=0}^\infty
|a_k|^2<\infty\bigr\}.
\]
We set,
\[
 \h_2(X)=\bigl\{\phi=\sum_{k=0}^\infty a_k\phi_k\, \bigl|\,  \sum_{k=0}^\infty
\la_k^2|a_k|^2<\infty  \bigr\}.
\]
We have,
\[
A^{(0,0)}(X)\subseteq \h_2(X)\subseteq \h_0(X).
\]
The right-hand side inclusion is obvious, and the other one can be deduced from \cite[§ 14.2.2 p.367]{Buser}. Since $\h_0(X)$
is complete with respect to $\vc$. Notice that $\h_2(X)$ is the  completion of $A^{(0,0)}(X)$ with respect
to  the following norm $\vc_2$, defined as follows: $\|\phi\|_2^2=\|\phi\|^2+\|\Delta \phi\|^2$, for any $\phi\in
A^{(0,0)}(X)$.\\

One can view $\h_2(X)$ in a more intrinsic way, namely as the set of $\phi\in \h_0(X)$, such that there exists $(\phi_j)_{j\in \N}$, a sequence in
$A^{(0,0)}(X)$ converging to
$\phi$ with respect to $L^2$-norm and such that the
sequence $(\Delta \phi_j)_{j\in \N}$ admits a limit in
$\h_0(X)$. So, we  can write,
\[
\h_2(X)=(I+\Delta)^{-1}\h_0(X).
\]

Let $\phi \in \h(X)$, by definition there exists $(\phi_j)_{j\in \N}$ a sequence in $A^{(0,0)}(X)$ converging to $\phi$ and
such that $(\Delta\phi_j)_{j\in\N}$ admits a limit. We can check easily that the limit is unique. This previous point  motivates
the following
definition: We let $Q$ be the operator on $\h_2(X)$ given by $Q(\phi)=\lim_{j\in \N}\Delta\phi_j$ for any $\phi\in \h_2(X)$
and $(\phi_j)_{j\in \N}$ described as above. Then $Q$ is a maximal positive and  selfadjoint extension of $\Delta$ with domain $\mathrm{Dom}(Q)=\h_2(X)$. We claim that,
\begin{equation}\label{Qfi}
Q(\phi)=\psi-\phi,
\end{equation}
for any $\phi\in \h_2(X)$, where $\psi$ is the unique element in $\h_0(X)$ such that $\phi=(I+\Delta)^{-1}\psi$. Let us
verify the last claim;  Let $\phi\in \h_2(X)$, since $I+\Delta$ is invertible then there exists a unique $\psi\in \h_0(X)$ such
that $\phi=(I+\Delta)^{-1}\psi$. Now, let
$(\psi_j)_{j\in \N}$ be
a sequence in $A^{(0,0)}(X)$ converging to $\psi$ with respect to the $L^2$-norm, it follows that
 \[\bigl(\phi_j:=(I+\Delta)^{-1}\psi_j\bigr)
_{j\in \N}\]
 converges to $\phi$ with respect to the $L^2$-norm. Since, $Q(\phi_j)=\Delta\phi_j=\psi_j-\phi_j$ for any $j\in \N$, hence
 $(Q(\phi_j))_{j\in\N}$
 converges to $\psi-\phi$. Then,
 \[
 Q(\phi)=\psi-\phi,
 \]

If $T$ is an extension of $\Delta$, that is a
linear selfadjoint  operator $T:\mathrm{Dom}(T)\lra
\h_0(X)$ such that: $\h_2(X)\subseteq\mathrm{Dom}(T)$ and
the restriction of $T$ to $\h_2(X)$ is $Q$. Pick $\phi=\sum_{j=0}^\infty a_j \phi_j$  in
$\h_0(X)$, then there exists $b_j\in \C$ for any $j\in \N$ such that:
\[
T\phi=\sum_{j=0}^\infty b_j \phi_j,
\]
We have,
\[
b_j=(T(\phi),\phi_j)=(\phi,T(\phi_j))=(\phi,Q(\phi_j))=\la_j(\phi,\phi_j)=\la_j a_j.
\]

Recall that $\|T\phi\|^2=\sum_{j=0}^\infty |b_j|^2<\infty $. From this we infer, $\sum_{j=0}^\infty \la_j^2|a_j|^2<\infty$, hence
$\phi\in \h_2(X)$. We conclude that $T=Q$. We say
$Q$ is a maximal selfadjoint extension of $\Delta$.\\

 Our goal now is to construct $Q_\infty$, a maximal selfadjoint extension of the operator $\Delta_{X,\infty}$. Let
 $(h_{X,u})_{u>1}$ be as before and let $(\Delta_{X,u})_{u>1}$ be the sequence of the associated Laplacians.

From \eqref{deltacompactoperateur},   the sequence $\bigl((I+\Delta_{X,u})^{-1} \bigr)_{u>1}$ converges to $(I+\Delta_{X,\infty})^{-1}$, with respect to one $L^2_{X,u_0}$-norm hence
 for  any $L^2_{X,v}$-norm with $v$ fixed.  We have also that  $(I+\Delta_{X,\infty})^{-1}$ is a compact linear operator on $\h_0(X)$. Note that $(I+\Delta_{X,u})^{-1}\h_0(X)$
 does not depend on $u$. Indeed,  this follows from the
 previous  discussion, and the fact that the metrics are uniformly equivalent. Then,
  \[\h_2(X)=(I+\Delta_{X,\infty})^{-1}\h_0(X).\]

  \begin{claim}\label{suiteconverge}
  Let $H$ be a Hilbert space. Let $(\vc_u)_{u\geq 1}$ be a sequence of uniformly equivalent hilbertian norms on $H$, converging
  to $\vc_\infty$, a Hilbert norm on $H$.

  Let $(\eta_u)_{u\geq 1}$ and $(\eta'_u)_{u\geq 1}$ two sequences in $H$, converging respectively to $\eta_\infty$ and
  $\eta'_\infty$ with respect, hence any norm $\vc_v$ with $v\geq 1$. Then, the complex sequence $\bigl((\eta_u,\eta'_u)_u \bigr)_{u\geq 1}$ converges
  to  $(\eta_\infty,\eta'_\infty)_\infty$.
    \end{claim}

  \begin{proof}
  We have,
  \[ (\eta_u,\eta'_u)_u-(\eta_\infty,\eta'_\infty)_\infty=(\eta_u-\eta_\infty,\eta'_u)_u+(\eta_\infty,\eta'_u-\eta_\infty')_u+(\eta_\infty,\eta'_\infty)_u-(\eta_\infty,\eta'_\infty)_\infty.
  \]
  Now, using the assumptions, there exists a constant $M$ such that $|(\eta_u-\eta_\infty,\eta'_u)_u+(\eta_\infty,\eta'_u-\eta_\infty')_u |\leq M \bigl(
  \|\eta_u-\eta_\infty\|_\infty+\|\eta_u'-\eta_\infty'\|_\infty\bigr)$, and since $(\vc_u)_{u\geq 1}$ converges to $\vc_\infty$ we
  conclude that $\bigl( (\eta_\infty,\eta'_\infty)_u\bigr)_{u\geq 1}$ converges to $(\eta_\infty,\eta'_\infty)_\infty$.

  \end{proof}

Let $\phi\in \h_2(X)$. There exists a $\psi \in \h_0(X)$ such that
 $\phi=(I+\Delta_{X,\infty})^{-1}\psi$. We claim that $\psi$ is unique and we will prove the uniqueness later (see the proof of
 \eqref{invertibleoperator}). Then we define $Q_\infty$, an extension of $\Delta_{X,\infty}$  as follows: Let $\phi\in \h_2$,
 so, by assumption,  there exists a unique $\psi\in \h_0(X) $ such that $\phi=(I+\Delta_{X,\infty})^{-1}\psi$,  we set:
 \[
 Q_\infty(\phi):=\psi-\phi,
 \]
Let us check that $Q_\infty$ is a positive selfadjoint extension of $\Delta_{X,\infty}$.  To establish the positivity of $Q_\infty$ we need the following claim: There exists a sequence $(\phi_u)_{u>1}$ in $\h_2(X)$ such
that $(\phi_u)_{u>1}$ converges to $\phi$ with respect to any $L^2$-norm, and such that $(Q_u(\phi_u))_{u>1}$ converges to
$Q_\infty(\phi)$. Indeed, Let $\phi_u:=(I+\Delta_{X,u})^{-1}\psi$, $\forall u>1$. We have
 \[
  \bigl((I+\Delta_{X,u})^{-1}\psi\bigr)_{u\geq 1}\xrightarrow[u\mapsto \infty]{} (I+\Delta_{X,\infty})^{-1}\psi\quad\text{see}\;
  \eqref{deltacompactoperateur}
 \]

Note that $Q_u(\phi_u)=\psi -\phi_u$ (see \eqref{Qfi}), which   converges to $\psi-\phi=Q_\infty(\phi)$ . Now, recall that $Q_u$ is
a positive
operator with respect to $(,)_u$. Namely, $(Q_u(\phi_u),\phi_u)_u\geq 0$. Since  $\bigl((,)_u\bigr)_{u\geq 1}$ converges
uniformly to $(,)_\infty$ and according to the previous claim \eqref{suiteconverge}. We conclude that
\[
(Q_\infty(\phi),\phi)_\infty\geq 0.
\]
Using the same argument, we prove that $Q_\infty$ is selfadjoint.\\

 Let $\phi\in A^{(0,0)}(X)$. By \eqref{laplaceTX}, the following element $\psi:=(I+\Delta_{X,\infty})\phi$ belongs
 to $ \h_0(X)$
, and
 \[
 (I+\Delta_{X,u})^{-1}\psi\xrightarrow[u\mapsto \infty]{} (I+\Delta_{X,\infty})^{-1}\psi,
 \]
 therefore,
 \[
 Q_\infty(\phi)=\psi-\phi=\Delta_{X,\infty} \phi.
 \]

Let $T$ be an extension of $Q_\infty$, that is a positive selfadjoint linear operator $T:\mathrm{Dom}(T)\lra \h_0(X)$ such
that $\h_2(X)\subseteq \mathrm{Dom}(T)$ and $T_{|_{\h_2(X)}}=Q_\infty$. Let $\phi\in \mathrm{Dom}(T)$, put $\psi:=(I+T)\phi$. We
have $\psi\in \h_0(X)$, thus $\theta:=(I+\Delta_{X,\infty})^{-1}\psi\in \h_2(X)$. Hence,
\[
(I+T)(\theta)=\theta+Q_\infty(\theta)=\theta+(\psi-\theta)=\psi.
\]
But recall that $\psi=(I+T)\phi$, then
\[
(T+I)(\theta-\phi)=0.
\]
Since $T$ is a positive operator, and so is $T+I$, it follows that
\[
\phi=\theta=(I+\Delta_{X,\infty})^{-1}\psi.
\]
Therefore,
\[
\mathrm{Dom}(T)=\h_2(X)\quad \text{and}\quad T=Q_\infty.
\]
So $Q_\infty$ is a maximal positive selfadjoint extension for $\Delta_{X,\infty}$.

\begin{theorem}\label{invertibleoperator} The operator $\Delta_{X,\infty}$ admits a maximal selfadjoint extension to $\h_2(X)$, we denote this extension
also by $\Delta_{X,\infty}$. We have:
\[
(I+\Delta_{X,\infty})(I+\Delta_{X,\infty})^{-1}=I,
\]
on $\h_0(X)$, where $I$ is the operator identity of $\h_0(X)$.
\[
(I+\Delta_{X,\infty})^{-1}(I+\Delta_{X,\infty})=I,
\]
on $\h_2(X)$, where   $I$ is the operator identity of $\h_2(X)$.
\end{theorem}

\begin{proof}
The first assertion follows from the previous discussion.

Recall that we supposed there exists a unique $\psi\in \h_0(X)$ such that $\phi=(I+\Delta_{X,\infty})^{-1}\psi$. Let us prove this
point. It suffices to prove the following:
\[
(I+\Delta_{X,\infty})(I+\Delta_{X,\infty})^{-1}=I,
\]
on $\h_0(X)$.\\

We fix $\xi\in A^{(0,0)}(X)$. We have,
\[
\underset{u\mapsto \infty}{\lim}\bigl\| \Delta_{{}_{X,u}}\xi \bigr\|^2_{L^2,u}=\bigl\|
\Delta_{{}_{X,\infty}}\xi
\bigr\|^2_{L^2,\infty}<\infty\quad \text{see}\; \eqref{laplaceTX},
\]
and
\[
 \biggl\| \frac{\pt \Delta_{{}_{X,u}}}{\pt u}\xi \biggr\|_{L^2,u}\leq \delta_X(u)\bigl\|\Delta_{{}_{X,u}}\xi\bigr\|_{L^2,u}\quad
 \text{see}\; \eqref{bornelapbelt}.
\]
We deduce, (for fixed $\xi$), there exists $C$, a constant  such that:
\[
 \biggl\| \frac{\pt \Delta_{{}_{X,u}}}{\pt u}\xi \biggr\|_{L^2,u}\leq C\delta_X(u),
\]
for $u\gg 1$.  Remember that the different norms $L^2_{X,u}$ are uniformly equivalent, so we can find $C'$ such that:
\[
 \biggl\| \frac{\pt \Delta_{{}_{X,u}}}{\pt u}\xi \biggr\|_{L^2,\infty}\leq C'\delta_X(u),
\]
Therefore,
\[
\bigl\| \Delta_{{}_{X,p}}\xi-\Delta_{{}_{X,q}}\xi \bigr\|_{L^2,\infty}\leq C'\int_p^q\delta_X(u)du,
\]
Thus $(\Delta_{{}_{X,p}}\xi)_{p\in \N}$ converges to $\Delta_{{}_{X,\infty}}\xi$ with respect to
$L^2_{X,\infty}$.\\

Now, let $\psi\in \h_0(X)$ and $\xi\in A^{(0,0)}(X)$. Using \eqref{suiteconverge}, we have
\begin{align*}
\bigl((\Delta_{X,\infty}+I)(\Delta_{X,\infty}+I)^{-1}\psi,\xi
\bigr)_{L^2,\infty}&=\bigl((\Delta_{X,\infty}+I)^{-1}\psi, (\Delta_{X,\infty}+I)\xi
\bigr)_{L^2,\infty}\\
&=\lim_{u\mapsto \infty}\bigl((\Delta_{X,u}+I)^{-1}\psi, (\Delta_{X,u}+I)\xi
\bigr)_{L^2,u}\\
&=\lim_{u\mapsto \infty}\bigl(\psi, \xi  \bigr)_{L^2,u}\\
&=\bigl(\psi, \xi  \bigr)_{L^2,\infty}.
\end{align*}
So, we have proved that for any $\psi\in \h_0(X)$,
\[
\bigl((\Delta_{X,\infty}+I)(\Delta_{X,\infty}+I)^{-1}\psi-\psi,\xi
\bigr)_{L^2,\infty}=0\quad \forall \, \xi\in A^{(0,0)}(X).
\]
To conclude, recall that if $D$ is dense linear subspace of a Hilbert space $(H,(,)_H)$, and suppose
there exists $v\in H $ such
that $(v,z)_H=0$ for all $z\in D$, so $v=0$. Indeed, take  $(z_j)_{j\in \N}$ a sequence in $D$
converging to $v$. We have
$(v,v)_H=\lim_{j\mapsto \infty}(v,z_j)_H=0$.\\

Applying this claim to $H=\h_0(X)$, $D=A^{(0,0)}(X)$ and $v=(\Delta_{X,\infty}+I)(\Delta_{X,\infty}+I)^{-1}\psi-\psi$.
It follows that,
\[
(\Delta_{X,\infty}+I)(\Delta_{X,\infty}+I)^{-1}=I.
\]
on $\h_0(X)$.\\

Now, let us prove the last assertion of the theorem. Let $\xi \in \h_2(X)$ and $\psi \in \h_0(X)$, we have
\begin{align*}
\bigl((\Delta_{X,\infty}+I)^{-1}(\Delta_{X,\infty}+I)\xi,\psi
\bigr)_{L^2,\infty}&=\bigl((\Delta_{X,\infty}+I)\xi, (\Delta_{X,\infty}+I)^{-1}\psi
\bigr)_{L^2,\infty}\\
&=\bigl(\xi, (\Delta_{X,\infty}+I)(\Delta_{X,\infty}+I)^{-1}\psi
\bigr)_{L^2,\infty} \\
&=\bigl(\xi, \psi  \bigr)_{L^2,\infty}.
\end{align*}
Then,
\[
(\Delta_{X,\infty}+I)^{-1}(\Delta_{X,\infty}+I)=I,
\]
on $\h_0(X)$.\\

\end{proof}

\begin{Corollaire}
  $\Delta_{{}_{X,\infty}}$ has an infinite positive discrete spectrum.
\end{Corollaire}
\begin{proof}
The existence of the spectrum and  its  nature  follows, for instance, from \cite[theorem 3.4 p.429]{functional}. The positivity is a
consequence of the positivity of the operator $\Delta_{{}_{X,\infty}}$.
\end{proof}

\begin{theorem}
$\Delta_{{}_{X,\infty}}$ admit a heat kernel, we denote it by  $e^{-t\Delta_{{}_{X,\infty}}}$, $t>0$.
\end{theorem}

\begin{proof}
We proved that $\Delta_{{}_{X,\infty}}$ is a positive selfadjoint operator. Then from \eqref{semi}, we deduce that $\Delta_{{}_{X,\infty}}$ generates a  semi-group $e^{-t\Delta_{{}_{X,\infty}}}$ for any $t>0$.
\end{proof}

\begin{proposition}
There exists  $(h_{X,u})_{u>1}$ a family of  smooth hermitian metrics on $TX$ converging uniformly to $h_{X,\infty}$, such
that  for any fixed
$t>0$:
\[
\biggl\|\frac{\partial}{\partial u}e^{-t\Delta_{X,u}}\biggr\|_{L^2,\infty}=O\bigl(\delta_X(u)\bigr)\quad \forall \, u\gg 1.
\]
\end{proposition}

\begin{proof}
We have,
\[
\begin{split}
\frac{\partial }{\partial u}e^{-t\Delta^u}&=\int_{0}^t e^{-(t-s)\Delta^u}\bigl((\partial_u\log h_{X,u}) \Delta^u\bigr)e^{-s\Delta^u}ds\quad \text{by}\; \eqref{heatkernel}\\
&=-\int_{0}^t e^{-(t-s)\Delta^u}\bigl(\partial_u\log h_{X,u} \bigr)\pt_se^{-s\Delta^u}ds.\\
\end{split}
\]
 For fixed $u>0$, let $\bigl(\phi_{u,k}\bigl)_{k\in \N}$ be an orthonormal basis  with respect of  $L^2_{X,u}$, as
an example, we can choose a set of eigenvectors of  $\Delta_{{}_{X,u}}$. Let  $\xi\in A^{(0,0)}(X)$,   there exist real $a_k$ for any $k\in \N$ such that $\xi=\sum_{k\in \N}a_{u,k} \phi_{u,k}$. We have,
\[
\begin{split}
\partial_t e^{-t\Delta_{{}_{X,u}}}\xi&=-\sum_{k\in\N} a_{u,k}\lambda_{u,k} e^{-\lambda_{u,k} t}\phi_{u,k}\\
&=-\frac{1}{t}\sum_{k} a_{u,k}t\lambda_{u,k} e^{-\lambda_{u,k}t}\phi_{u,k},
\end{split}
\]

Since $a^2 e^{-a}\leq 4e^{-2}, \forall a\geq 0 $, it follows that
\[
\begin{split}
\Bigl\|\partial_t e^{-t\Delta_{{}_{X,u}}} \xi \Bigr\|^2_{L^2,u}&=\frac{1}{t^2}\sum_{k\in \N} a_{u,k}^2\bigl(\lambda_{u,k} t\bigr)^2e^{-2\lambda_{u,k}t }\\
&\leq \frac{e^{-2}}{t^2}\sum_{k\in \N}a_{u,k}^2\\
&=\frac{e^{-2}}{t^2}\bigl\|\xi\bigr\|_{L^2,u}^2,
\end{split}
\]
Therefore,
\[
\Bigl\|\partial_t e^{-t\Delta_{{}_{X,u}} }\Bigl\|_{L^2,u}\leq \frac{e^{-1}}{t} \quad \forall\, t>0.
\]

Using the fact that all the norms are uniformly  equivalent, we can find two  constants $M_1$ and $M_2$ which are  independent of $u$ such that:
\[
\Bigl\|\partial_te^{-t\Delta_{{}_{X,u}}}\Bigr\|_{L^2,\infty}\leq \frac{M_1}{t} \quad \forall\, u,\;\forall \, t>0.
\]
and
\[
\bigl\| e^{-s\Delta_{X,u}}\bigr\|_{L^2,\infty}\leq M_2\quad \forall\, s\geq 0.
\]
Fix $t>0$. for any $u\gg 1$
\begin{align*}
\biggl\|\int_{\frac{t}{2}}^t e^{-(t-s)\Delta^u}\bigl((\partial_u\log h_{X,u}) \bigr)\partial_se^{-s\Delta^u}ds\biggr\|_{L^2,\infty}&\leq \int_{\frac{t}{2}}^t\Bigl\| e^{-(t-s)\Delta^u}\bigl((\partial_u\log h_{X,u}) \bigr)\partial_se^{-s\Delta^u} \Bigr\|_{L^2,\infty}ds \\
&\leq \int_{\frac{t}{2}}^t\delta_{X}(u) \bigl\| e^{-(t-s)\Delta_{X,u}}\bigr\|_{L^2,\infty} \bigl\| \pt_s e^{-s\Delta_{X,u}}\bigr\|_{L^2,\infty}\\
&\leq M_1 M_2\delta_X(u)\int_{\frac{t}{2}}^t \frac{1}{s}ds\\
&=O(\delta_X(u))
\end{align*}

Using an integration by parts, we get
{\allowdisplaybreaks
\begin{align*}
\int_{0}^{\frac{t}{2}} &e^{-(t-s)\Delta^u}\bigl((\partial_u\log h_{X,u}) \bigr)\partial_se^{-s\Delta^u}ds\\ &=\Bigl[e^{-(t-s)\Delta_{{}_{X,u}}}(\pt_u\log h_{X,u})e^{-s\Delta_{{}_{X,u}}} \Bigr]_0^{\frac{t}{2}}
-\int_0^{\frac{t}{2}}\pt_s( e^{-(t-s)\Delta_{{}_{X,u}}})\bigl(\pt_u\log h_{X,u}\bigr) e^{-s\Delta_{{}_{X,u}}}ds     \\
&=e^{-\frac{t}{2}\Delta_{{}_{X,u}}}\bigl(\pt_u \log h_{X,u}\bigr) e^{-\frac{t}{2}\Delta_{{}_{X,u}}}-e^{-t\Delta_{{}_{X,u}}}\bigl(\pt_u \log h_{X,u}\bigr) I- \int_{0}^\frac{t}{2} \partial_s (\mathcal{O}^{-(t-s)\Delta^u})\bigl(\partial_u\log h_{X,u} \bigr)e^{-s\Delta^u}ds\\
&=e^{-\frac{t}{2}\Delta_{{}_{X,u}}}(\pt_u \log h_{X,u}) e^{-\frac{t}{2}\Delta_{{}_{X,u}}}-e^{-t\Delta_{{}_{X,u}}}(\pt_u \log h_{X,u}) I+ \int_{\frac{t}{2}}^t \partial_s (\mathcal{O}^{-s\Delta^u})\bigl(\partial_u\log h_{X,u} \bigr)e^{-(t-s)\Delta^u}ds,
\end{align*}}

Then there exists a constant $M_3$ such that:
\[
\begin{split}
\biggl\|\int_{0}^{\frac{t}{2}} e^{-(t-s)\Delta^u}\bigl((\partial_u\log h_{X,u})
\bigr)\partial_se^{-s\Delta^u}ds\biggr\|_{L^2,\infty}&\leq  M_3\delta_X(u).
\end{split}
\]
We conclude that:
\begin{equation}\label{key}
\biggl\|\frac{\partial}{\partial u} e^{-t\Delta_{{}_{X,u}}}\biggr\|_{L^2,\infty}=O\bigl(\delta_X(u)\bigr)\quad \forall\, u\gg 1.
\end{equation}

\end{proof}

As an application of the previous results, we show that  $e^{-t\Delta_{X,\infty}}$ can be approximate by a
sequence of heat kernels
associated to smooth metrics.
\begin{theorem}\label{variationEu}
 for any $t>0$, we have:
\[
 \bigl(e^{-t\Delta_{{}_{X,u}}}\bigr)_u\xrightarrow[u\to+\infty]{}e^{-t\Delta_{X,\infty}},
\]
In particular, $e^{-t\Delta_{{}_{X,\infty}}}$ is a compact operator from $\h_0(X)$ to $\h_0(X)$.
\end{theorem}

\begin{proof}
By \eqref{key} and \eqref{operateurcompact}, the sequence $\bigl(e^{-t\Delta_{{}_{X,u}}}\bigr)_{u>1}$ converges to
a limit, which will be denoted  by $L_t$, for any $t>0$. According to  \eqref{operateurcompact}, $L_t$ is a compact operator.

 By assumption, we may assume that $1-\frac{h_{X,\infty}}{h_{X,u}}=O(\frac{1}{u})$ for any  $u\gg 1$. On the other hand, we have $\Delta_{X,\infty}=\frac{h_{X,\infty}}{h_{X,u}}\Delta_{X,u}$. Then
\[
 \begin{split} \bigl(\partial_t+\Delta_{{}_{X,\infty}}\bigr)e^{-t\Delta_{{}_{X,u}}}&=\bigl(\partial_t+\frac{h_{X,\infty}}{h_{X,u}}\Delta_{{}_{X,u}}\bigr)e^{-t\Delta_{{}_{X,u}}}\\
&=\Bigl(1-\frac{h_{X,\infty}}{h_{X,u}}\Bigr)\,\partial_te^{-t\Delta_{{}_{X,u}}}\\
&=O\bigl(\frac{1}{u}\bigr)\,\partial_t e^{-t\Delta_{{}_{X,u}}}.
 \end{split}
\]

If we  fix  $t_0>0$, we have  shown that $\|\partial_t e^{-t\Delta_{{}_{X,u}}}\|_{L^2,\infty}$ is bounded for any $t\geq
t_0$.   It follows that
\[
 \bigl(\partial_t+\Delta_{{}_{X,\infty}}\bigr)L_t=0 \quad \forall\, t\geq t_0.
\]
Moreover $L_t\rightarrow I$ when $t\mapsto 0$. Since  $e^{-t\Delta_{{}_{X,\infty}}}$ satisfies the same properties, and
by the  uniqueness  of the heat kernel solution, it follows that $L_t=e^{-t\Delta_{{}_{X,\infty}}}$, for any $t>0$.
\end{proof}

\subsection{The trace and the Zeta function associated to $\Delta_{X,\infty}$}\label{tracezetainfty}

For any $u\in [1,\infty]$, we  consider the norm $L^2_{X,u}$ on
$A^{(0,0)}(X)$, we recall that this norm is induced by
$h_{X,u}$ and $h_{\mathcal{O}}$ (a constant metric on $\mathcal{O}$). An operator $T$ on the
completion  of $A^{(0,0)}(X)$ with respect to
$L^2_{X,u}$ is said to be of trace-class if the sum
$\sum_{k\in \N}\bigl(T\xi_{u,k},\xi_{u,k} \bigr)_{L^2,u},
$ is absolutely convergent for one, hence any orthonormal
basis $\bigl(\xi_{u,k}\bigr)_{k\in \N}$. The value
of this sum, which is independent of the choice of the
basis, is called the trace of $T$, and will be denoted
here by $\mathrm{Tr}_u(T)$.\\

For any $u\in [1,\infty[$, we consider the operator $P^u e^{-t \Delta_{X,u}}$, where $P^u$  is the orthogonal projection with
kernel  equal to  $H^0(X,\mathcal{O})$, with respect to $L^2_{X,u}$. We will need the following lemma, which describes  the variation of
$(P^u)_{u\geq
1}$ with respect to $u$ sufficiently large:
\begin{lemma}\label{variationPu}
We have, $P^u$ is a
bounded operator and
\[
\frac{\pt P^u}{\pt u}=O(\delta_X(u)),\quad \forall u\gg 1.
\]
\end{lemma}
\begin{proof}
 Let $1< u<\infty$. By definition of $P^u$, we have for any $\xi\in A^{0,0}(X)$, there exists
$a^{(u)}(\xi)$ a  complex number, such that
\[
P^u\xi=\xi+a^{(u)}(\xi)\footnote{We view $a^{(u)}(\xi)$ as an element in $A^{(0,0)}(X)$.}.
\]
Then $a^{(u)}(\xi)$ satisfies:
\[
  a^{(u)}(\xi)=-\frac{\bigl( \xi,1\bigr)_{L^2,u}}{\bigl(1,1\bigr)_{L^2,u}}.
\]
We have
\[
 \frac{\pt P^{(u)}}{\pt u}\xi=\frac{\pt a^{(u)}(\xi)}{\pt u}=-\bigl(1,1\bigr)_{L^2,u}^{-1}\frac{\pt}{\pt u}\bigl( \xi,1\bigr)_{L^2,u}+\bigl( \xi,1\bigr)_{L^2,u}\bigl(1,1\bigr)_{L^2,u}^{-2}\frac{\pt}{\pt u} \bigl( 1,1\bigr)_{L^2,u}.
\]
Since $(h_{X,u})_u$ converges to $h_{X,\infty}$ uniformly, then  $\Bigl(\bigl( 1,1\bigr)_{L^2,u}^{-1}\Bigr)_u$ is bounded. Then,
there exist  constants $m,m'>0$  such that:
\[
\bigl\|P^u\xi \bigr\|_{L^2,\infty}\leq \|\xi\|_{L^2,\infty}+m'\|\xi\|_{L^2,u},
\]
then $P^u$ is bounded with respect to $L^2_{X,\infty}$, and
\[
 \Bigl\|\frac{\pt P^{(u)}}{\pt u}\xi\Bigr\|_{L^2,\infty}\leq m\Bigl\|\frac{\pt}{\pt u}\bigl( \xi,1\bigr)_{L^2,u}\Bigr\|_{L^2,\infty}+m\|\xi\|_{L^2,u}\Bigl\|\frac{\pt}{\pt u} \bigl( 1,1\bigr)_{L^2,u}\Bigr\|_{L^2,u}\quad \forall\, u>1.
\]

Let  $\xi,\eta\in A^{(0,0)}(X)$, we have
\begin{align*}
 \Bigl|\frac{\pt}{\pt u}\bigl(\xi,\eta \bigr)_{L^2,u}\Bigr|&=\Bigl|\int_X \xi\,\overline{\eta}\frac{\pt }{\pt
 u}\bigl(\omega_{X,u}\bigr)\Bigr|\\
&=\Bigl|\int_X  \xi\,\overline{\eta}\Bigl(\frac{\pt }{\pt u}\log h_{X,u}\Bigr)\omega_{X,u}\Bigr|\\
&=\Bigl|\Bigl( \xi,\bigl(\frac{\pt}{\pt u}\log h_{X,u}\bigr)\eta\Bigr)_{L^2,u}\Bigr|\\
&\leq \delta_X(u) \bigl\|\xi\bigr\|_{L^2,u}\bigl\|\eta\bigr\|_{L^2,u}\quad \text{by Cauchy-Schwartz inequality}.
\end{align*}
Recall that the metrics are uniformly equivalent, we deduce there exists a constant $m''$ such that:
\[
 \Bigl\|\frac{\pt P^{(u)}}{\pt u}\xi\Bigr\|_{L^2,\infty}\leq m''\,\delta_X(u) \bigl\|\xi\bigr\|_{L^2,\infty}\quad \forall\, u>1.
\]
Then,
\[
 \frac{\pt P^u}{\pt u}=O(\delta_X(u)),\quad u\gg 1.
\]
\end{proof}
We need also this technical lemma:
 \begin{lemma}\label{liminfsomme}
  Let $\{c_{n,i}: n\in \N, i\in N \}$ be a family of positive real. We have
  \[
 \underset{n\mapsto \infty}{\liminf}\sum_{i}c_{n,i}\geq \sum_i \underset{n\mapsto \infty}{\liminf}c_{n,i}.
\]

 \end{lemma}
\begin{proof}
 Let $N$ be an nonzero integer. We have
\[
 \sum_{i=1}^\infty c_{k,i}\geq \sum_{i=1}^N\inf_{l\geq n} c_{l,i},\quad \forall n\;\forall k\geq n
\]
then,
\[
 \inf_{k\geq n}\sum_{i=1}^\infty c_{k,i}\geq \sum_{i=1}^N\inf_{l\geq n} c_{l,i},\quad \forall n
\]
then,
\[
\underset{n}{ \liminf}\sum_i c_{n,i}\geq \sum_{i=1}^N\underset{n}{\liminf} c_{n,i}.
\]
Since all the terms are positive, we conclude that,
\[
\underset{n}{ \liminf}\sum_i c_{n,i}\geq \sum_{i}\underset{n}{\liminf} c_{n,i}.
\]
\end{proof}
We recall that when $u\geq 1$, the metric  $h_{X,u}$ is smooth, then  the spectral theory of generalized
Laplacians  states that the operator $P^u e^{-t \Delta_{X,u}}$ is  of trace-class,  for any $t>0$. One defines
the so-called Theta function given by $\theta_{X,u}(t)=\mathrm{Tr}_u(P^u e^{-t \Delta_{X,u}})$. If $(\la_{u,k})_{k\in \N}$ is the
set of eigenvalues of $\Delta_{X,u}$ counted with multiplicity and  in increasing order,  we have
\[
\theta_{X,u}(t)=\sum_{k\in \N^\ast}e^{-t\la_{u,k}}\quad \forall \, t>0,
\]
we refer the reader to chapter 2 of \cite{heat}.\\

The smoothness of the metrics is a necessary condition in order to use  the classical spectral theory of Laplacians. But, an
integrable metric may be singular, as in following example: Let $\p^1$ the complex projective line. Since
 $T\p^1\simeq \mathcal{O}(2)$, we can endow
$\p^1$ with the following metric:
\[
h_{\p^1,\infty}(s,s)([x_0:x_1]):=\frac{|s([x_0,x_1])|^2}{\max(|x_0|,|x_1|)^4}, \quad \forall \, [x_0,x_1]\in \p^1
\]
where $s$ is a local holomorphic section of $\mathcal{O}(2)$.  Then, we can show that $h_{\p^1,\infty}$ is a  singular integrable
metric.
 \begin{definition}
\rm{We set,
\[
 \theta_{X,\infty}(t):=\mathrm{Tr}\bigl(P^\infty e^{-t\Delta_{X,\infty}}\bigr)\quad \forall\, t>0,
\]
where $P^\infty$ is the orthogonal projection with
respect to $L^2_{X,\infty}$ with  $H^0(X,\mathcal{O})$ as kernel
.  We say $\theta_{X,\infty}$, is the Theta function
associated to $\Delta_{X,\infty}$. }
\end{definition}

\begin{theorem}\label{thetafinite}
 We have,
\[
 0\leq \theta_{X,\infty}(t)<\infty\quad \forall t>0,
\]
hence $P^\infty e^{-t\Delta_{X,\infty}}$ is a trace-class operator
for any $t>0$.
\end{theorem}
\begin{proof}


From \eqref{tracenorme1},
\begin{equation}\label{n1}
\Bigl|\mathrm{Tr}_\infty\Bigl(\bigl(\Delta_{X,u}+I \bigr)^{-2}\Bigr)\Bigr|\leq \Bigl\|\bigl(\Delta_{X,u}+I \bigr)^{-2}\Bigr\|_{1,\infty}\quad \forall \, u\geq 1.
\end{equation}

and by \eqref{equivtracenorme11}, we have for any $0<\eps\ll 1$ and any $u\gg 1$:

\begin{equation}\label{n2}
\frac{1-\eps}{1+\eps}\Bigl\|\bigl(\Delta_{X,u}+I \bigr)^{-2}\Bigr\|_{1,u}\leq \Bigl\|\bigl(\Delta_{X,u}+I
\bigr)^{-2}\Bigr\|_{1,\infty}\leq \frac{1+\eps}{1-\eps}\Bigl\|\bigl(\Delta_{X,u}+I \bigr)^{-2}\Bigr\|_{1,u}.
\end{equation}

Also, we have:
\begin{equation}\label{n3}
\Bigl\|\bigl(\Delta_{X,u}+I \bigr)^{-2}\Bigr\|_{1,u}=\sum_{k\in \N }\frac{1}{(\la_{u,k}+1)^2}\leq \zeta_{X,u}(2)+1<\infty\quad
\forall\,u\geq 1,
\end{equation}
where $(\la_{u,k})_{k\in \N}$ is the set of eigenvalues of $\Delta_{X,u}$, ordered in increasing
order, and $\zeta_{X,u}$ is the Zeta function attached
to $\Delta_{X,u}$ which is finite on the set $\{s\in \C\, |\, \mathrm{Re}(s)>1 \}$, see \cite[§ 9.6]{heat}.

From \eqref{n1}, \eqref{n2} and \eqref{n3}, we conclude that  $\bigl(\Delta_{X,u}+I \bigr)^{-2} $ is
a  trace-class operator for the norm $L^2_{X,\infty}$, for any $u$ sufficiently large.
Hence $\mathrm{Tr}\bigl((\Delta_{X,u}+I \bigr)^{-2}\bigr)$ is finite.\\

By the Lidskii theorem,   (cf. \eqref{lidskii}) and since $\bigl(\Delta_{X,u}+I \bigr)^{-2} $ is a trace-class operator for
$L^2_{X,\infty}$, then
$\mathrm{Tr}_\infty\bigl((\Delta_{X,u}+I \bigr)^{-2}\bigr)$ is the sum of its eigenvalues counted with their multiplicity. But they
 are positive, hence
\[
 \mathrm{Tr}_\infty\bigl((\Delta_{X,u}+I \bigr)^{-2}\bigr)=\Bigl\|(\Delta_{X,u}+I \bigr)^{-2}\Bigr\|_{1,u}.
\]

We have
\[
\begin{split}
 \frac{\pt}{\pt u}\bigl(\Delta_{X,u}+I \bigr)^{-2}&=-(\Delta_{X,u}+I)^{-2}\frac{\pt}{\pt u}\bigl(\Delta_{X,u}+I
\bigr)^2 \bigl(\Delta_{X,u}+I \bigr)^{-2}\\
&=-(\Delta_{X,u}+I)^{-2}\frac{\pt \Delta_{X,u}}{\pt u} \bigl(\Delta_{X,u}+I
\bigr)^{-1}-(\Delta_{X,u}+I)^{-1}\frac{\pt \Delta_{X,u}}{\pt u} \bigl(\Delta_{X,u}+I \bigr)^{-2},
\end{split}
\]
hence, for $u$ sufficiently large
\[
 \begin{split}
\biggl|\frac{\pt}{\pt u}\mathrm{Tr}_{{}_\infty}\Bigl(\bigl(\Delta_{{}_{X,u}}+I \bigr)^{-2} \Bigr) \biggr|&=\biggl|
\mathrm{Tr}_{{}_\infty}\Bigl((\Delta_{{}_{X,u}}+I)^{{}}\frac{\pt \Delta_{{}_{X,u}}}{\pt u} \bigl(\Delta_{{}_{X,u}}+I
\bigr)^{{}^{-1}} \Bigr)
+\mathrm{Tr}_{{}_\infty}\Bigl((\Delta_{X,u}+I)^{{}^{-1}}\frac{\pt \Delta_{X,u}}{\pt u} \bigl(\Delta_{X,u}+I \bigr)^{{}}
\Bigr)\biggr|\\
&=2\biggl| \mathrm{Tr}_\infty\bigl((\Delta_{{}_{X,u}}+I)^{{}}\frac{\pt \Delta_{{}_{X,u}}}{\pt u}
\bigl(\Delta_{X,u}+I \bigr)^{-1} \Bigr)\biggr|\quad\text{by}\;\eqref{commutetrace11}\\
&\leq 2\Bigl\| (\Delta_{{}_{X,u}}+I)^{{}}\frac{\pt \Delta_{{}_{X,u}}}{\pt u}
\bigl(\Delta_{X,u}+I \bigr)^{-1} \Bigr\|_{1,\infty}\quad \text{by}\;\eqref{tracenorme1} \\
&\leq c_2\delta_X(u)\Bigl\|\bigl(\Delta_{X,u}+I \bigr)^{-2}\Bigr\|_{1,\infty}  \quad
\text{by}\;\eqref{ineqinfty1}\;\text{and}\; \eqref{normetrace}\\
&\leq  c\,\delta_X(u) \Bigl\|\bigl(\Delta_{X,u}+I \bigr)^{-2}\Bigr\|_{1,u}\quad \text{the existence
of }\,c\;\text{follows from}\; \eqref{n2}.
\end{split}
\]
We have proved then the following:
\[
\biggl|\frac{\pt}{\pt u}\mathrm{Tr}_{{}_\infty}\Bigl(\bigl(\Delta_{{}_{X,u}}+I \bigr)^{-2} \Bigr) \biggr|\leq  c\,\delta_X(u) \,
\mathrm{Tr}_{{}_\infty}\Bigl(\bigl(\Delta_{{}_{X,u}}+I \bigr)^{-2} \Bigr).
\]

If we set $\al(u):=\mathrm{Tr}_\infty\bigl((\Delta_{X,u}+I \bigr)^{-2}\bigr)$, then the last inequality becomes:
\[
 \Bigl| \frac{\pt}{\pt u}\al(u)\Bigr|\leq c\, \delta_X(u) \al(u)\quad \forall\, u>1.
\]
Thus,
\[
\biggl|\log \biggl(\frac{\al(u)}{\al(u')}\biggr) \biggr|\leq c\,\Bigl|\int_u^{u'}\delta_X(v)dv\Bigr|\quad \forall\,
u>1.
\]
From \eqref{deltaXU}, we can choose $(h_{X,u})_{u\geq 1}$  such that : $\delta_X(u)=O(\frac{1}{u^2})$ for any $u\gg 1$. Then,
\[
\biggl|\log \biggl(\frac{\al(u)}{\al(u')}\biggr) \biggr|\leq
\Bigl|\int_u^{u'}O(\frac{1}{v^2})dv\Bigr|=O(\bigl|\frac{1}{u}-\frac{1}{u'}\bigr|)\quad \forall\,
u,u'\gg 1,
\]
It follows   that $u\mapsto \al(u)$ is bounded on a interval of the form $[A,\infty[$, and we choose $A>1$.\\

For any $t>0$, there exists $c_t$ a constant such that:
\[
 e^{-t a}\leq \frac{c_t}{(1+a)^2}\quad \forall a>0.
\]
If we denote by $\theta_{X,u}$ the Theta function associated to $\Delta_{X,u}$, then
\[
 \theta_{X,u}(t)=\sum_{k= 1}^\infty e^{-t\la_{u,k}}\leq c_t \sum_{k=0}^\infty\frac{1}{(\la_{u,k}+1)^2}=c_t\,\al(u)\quad \forall\, u\geq 1.
\]
In particular, this inequality holds  for any $u\geq A$. Since, $u\mapsto \al(u)$ is bounded on $[A,\infty[$, we
conclude that for any $t>0$ fixed, the following sequence:
\[\bigl(\theta_{X,u}(t)\bigr)_{u\geq A},\] is bounded.\\

To establish the theorem, that is $P^\infty e^{-t\Delta_{X,\infty}}$ is a trace-class operator  for any $t>0$, it suffices
to find  $0<\eps\ll 1$ such that:
\[
 \theta_{X,\infty}(t)\leq \frac{1+\eps}{1-\eps}\liminf_{u\mapsto \infty}\theta_{X,u}(t)\quad \forall\, t>1,
\]
and since the right-hand side is bounded, this  yields to:
\[
 \theta_{X,\infty}(t)<\infty\quad \forall\, t>0.
\]
So it remains  to prove there exists  $0<\eps\ll 1$ such that:
\[
 \theta_{X,\infty}(t)\leq \frac{1+\eps}{1-\eps}\liminf_{u\mapsto \infty}\theta_{X,u}(t)\quad \forall\, t>1.
\]
In order to prove this claim, we start by  comparing $\si_n(P^\infty e^{-t\Delta_{X,\infty}})_\infty$\footnote{See
\eqref{defvaleursinguliere}
for the  definition of  $\si_n(\cdot)$ } with $\si_n(P^u
e^{-t\Delta_{X,u}})_\infty$ for any $u\gg 1$.  Let  $R$ be a finite-rank operator such that its rank  $\leq n$.
We fix t>0, we have
\begin{align*}
 \si_n(P^\infty e^{-t\Delta_{X,\infty}})_\infty&\leq \bigl\|P^\infty e^{-t\Delta_{X,\infty}}-R\bigr\|_{L^2,\infty}\\
 &\leq \bigl\|P^\infty e^{-t\Delta_{X,\infty}}-P^\infty e^{-t\Delta_{X,u}}\bigr\|_{L^2,\infty}+\bigl\|P^\infty e^{-t\Delta_{X,u}}-P^u e^{-t\Delta_{X,u}}\bigr\|_{L^2,\infty}+ \bigl\|P^u e^{-t\Delta_{X,u}}-R\bigr\|_{L^2,\infty}\\
 &\leq \bigl\|P^\infty\|_{L^2} \bigl\|e^{-t\Delta_{X,\infty}}-e^{-t\Delta_{X,u}}\bigr\|_{L^2,\infty}+\bigl\|P^\infty -P^u\bigr\|_{L^2,\infty}\bigl\|e^{-t\Delta_{X,u}}\bigr\|_{L^2,\infty}+  \bigl\|P^u e^{-t\Delta_{X,u}}-R\bigr\|_{L^2,\infty}.
\end{align*}
Since the last inequality holds for an arbitrary $R$, then
\[
 \si_n(P^\infty e^{-t\Delta_{X,\infty}})_\infty\leq  \bigl\|P^\infty\|_{L^2} \bigl\|e^{-t\Delta_{X,\infty}}-e^{-t\Delta_{X,u}}\bigr\|_{L^2,\infty}+\bigl\|P^\infty -P^u\bigr\|_{L^2,\infty}\bigl\|e^{-t\Delta_{X,u}}\bigr\|_{L^2,\infty}+  \si_n(P^u e^{-t\Delta_{X,u}})_\infty.
\]
According to \eqref{variationEu} and \eqref{variationPu},  we know that  $(\mathcal{O}^{-t\Delta_{X,u}})_{u\geq 1}$ (resp. $(P^u)_{u\geq 1}$ ) converges  to $ e^{-t\Delta_{X,\infty}}$ (resp. to
$P^\infty$) with respect to the norm $L^2_{X,\infty}$, and that the real sequence $\bigl(\bigl\|e^{-t\Delta_{X,u}}\bigr\|_{L^2,\infty}\bigr)_u$ is bounded, hence the previous
inequality yields to:
\begin{equation}\label{SSSSSX}
\si_n(P^\infty e^{-t\Delta_{X,\infty}})_\infty\leq
 \liminf_{u\mapsto \infty}\si_n(P^u e^{-t\Delta_{X,u}})_\infty.
\end{equation}
Therefore,
\begin{align*}
\theta_{X,\infty}(t)&=\mathrm{Tr}_\infty(P^\infty e^{-t\Delta_{X,\infty}})\\
&=\sum_{n\in \N}\si_n(P^\infty e^{-t\Delta_{X,\infty}})_\infty\quad \text{see the definition} \;\eqref{definitionTrace}\\
&\leq \sum_{n\in \N} \liminf_{u\mapsto \infty}\si_n(P^u e^{-t\Delta_{X,u}})_\infty\quad \text{by}\,\eqref{SSSSSX}\\
&\leq \liminf_{u\mapsto \infty}\sum_{n\in \N} \si_n(P^u e^{-t\Delta_{X,u}})_\infty\quad \text{by}\,\eqref{liminfsomme}\\
&=\liminf_{u\mapsto \infty}\|P^u e^{-t\Delta_{X,u}}\|_{1,\infty}\quad \text{see the definition} \;\eqref{definitionTrace}.
\end{align*}
So,
\begin{equation}\label{n4}
\theta_{X,\infty}(t)\leq \liminf_{u\mapsto \infty}\|P^u e^{-t\Delta_{X,u}}\|_{1,\infty}.
\end{equation}
Let $0<\eps<1$, using  \eqref{equivtracenorme11}, we get
\begin{equation}\label{n5}
\frac{1-\eps}{1+\eps}\|P^u e^{-t\Delta_{X,u}}\|_{1,\infty}\leq \|P^u e^{-t\Delta_{X,u}}\|_{1,u} \leq \frac{1+\eps}{1-\eps}\|P^u e^{-t\Delta_{X,u}}\|_{1,\infty}\quad \forall\, u\gg 1.
\end{equation}
Recall that
\begin{equation}\label{n6}
 \theta_{X,u}(t)=\bigl\|P^u e^{-t\Delta_{X,u}}\bigr\|_{1,u}.
\end{equation}
Finally, taking into account \eqref{n4},\eqref{n5} and \eqref{n6} we obtain:
\[
\theta_{X,\infty}(t)\leq \frac{1+\eps}{1-\eps}\liminf_{u\mapsto \infty}\theta_{X,u}(t).
\]
We conclude that for any $t>0$, the operator $P^\infty e^{-t\Delta_{X,\infty}}$ is  of trace-class.
\end{proof}

We denote by $(\la_{\infty,k})_{k\in \N}$ the sequence of the eigenvalues of $\Delta_{X,\infty}$  counted with their multiplicity,
and ordered in increasing order.
\begin{theorem}\label{key2}
For any $t>0$ fixed, we have:
\[
\bigl(\theta_{X,u}(t)\bigr)_{u\geq 1}\xrightarrow[u\mapsto \infty]{} \theta_{X,\infty}(t),
\]
and,
\[
\zeta_{X,\infty}(s):=\frac{1}{\Gamma(s)}\int_0^\infty t^{s-1}\theta_{X,\infty}(t)dt=\sum_{k=1}^\infty \frac{1}{\la_{\infty,k}^s},
\]
is finite for any $s\in \C$, such that $\mathrm{Re}(s)>1$. This
function of $s$ admits a meromorphic continuation to the
whole complex plane with a pole at $s=1$, and it is holomorphic at $s=0$. We have,
\[
\zeta'_{X,\infty}(0)=\int_1^\infty \frac{\theta_{X,\infty}(t)}{t}dt+\gamma b_{\infty,-1}-b_{\infty,0}+\int_0^1 \frac{\rho_{X,\infty}(t)}{t}dt=\lim_{u\mapsto \infty} \bigl(\zeta'_{X,u}(0)\bigr)_{u\geq 1},
\]
where $b_{\infty,-1}$, $b_{\infty,0}$ are real numbers and $\rho_{X,\infty}$ is a real function  such that $\theta_{X,\infty}(t)=\frac{b_{\infty,-1}}{t}+b_{\infty,0}
+ \rho_{X,\infty}(t)$, and  $\rho_{X,\infty}(t)=O(t)$ for $t>0$ sufficiently small.
\end{theorem}
We introduce then the following definition:
\begin{definition}
\rm{Let $X$ be a compact Riemann surface endowed with a integrable metric, $h_{X,\infty}$. We set
\[T\bigl((X,\omega_{X,\infty}),(\mathcal{O},h_{\mathcal{O}}) \bigr):=\zeta'_{\Delta_{X,\infty}}(0),\] and we call it the holomorphic
analytic torsion attached to $\bigl((X,\omega_{X,\infty}),(\mathcal{O},h_{\mathcal{O}}) \bigr)$, where $h_\mathcal{O}$ is a constant
metric.}
\end{definition}

Before to prove this theorem, recall that  for  any $u>1$ and any integer $k$, there exist real numbers
$a_{u,-1},a_{u,0},\ldots,a_{u,k}$ such that:
\[
\theta_{X,u}(t)=\sum_{i=-1}^k a_{u,i}t^i+O(t^{k+1}),
\]
for any $t$ small enough, see \cite[p. 94]{heat}. Notice that, $a_{u,-1}=4\pi\, \mathrm{vol}_u(X)$, see  \cite[theorem 2.41]{heat}.

\begin{proposition}\label{a_uuu}
We consider $(h_{X,u})_{u>1}$ as before. Then  $(a_{u,0})_{u\geq 1}$ and $(a_{u,-1})_{u\geq 1}$ converge respectively to finite
limit when $u$ goes to $\infty$.
\end{proposition}

\begin{proof}
We know that:
\[
 a_{u,-1}=4\pi \,\mathrm{vol}_u(X).
\]
Since, the sequence of hermitian norms, $(L^2_{X,u})_{u>0}$ tends to $L^2_{X,\infty}$. It follows that $\lim_{u \mapsto \infty} a_{u,-1}=4\pi \,\mathrm{rg}(\mathcal{O})\,\mathrm{vol}_\infty(X)=:a_{\infty,-1}$.

Let us prove the second assertion. Let $u>1$, and $t$ be a positive real, and we consider  $\bigl((TX,th_{X,u}); (\mathcal{O},h_{\mathcal{O}})\bigr)$.
The  variation of   Quillen metrics associated to $t$ is given by the following anomaly formula,
see \cite{BGS1}:
\[
\begin{split}
 -\log h_{Q,((TX,th_{X,u}); (\mathcal{O},h_{\mathcal{O}}))}+\log h_{Q,((TX,h_{X,u}); (\mathcal{O},h_{\mathcal{O}}))}&=\int_X ch(\mathcal{O},h_{\mathcal{O}})\widetilde{Td}(TX,th_{X,u},h_{X,u}),\\
\end{split}
\]
One checks, using the local expression of the Laplacian, that:
\[
\Delta_{t,u}:= \Delta_{((TX,th_{X,u}); (\mathcal{O},h_{\mathcal{O}})) }=t^{-1}\Delta_{((TX,h_{X,u}); (\mathcal{O},h_{\mathcal{O}}))}=t^{-1}\Delta_{1,u}\quad \forall\,t>0.
\]
It follows that $\zeta_{\Delta_{t,u}}'(0)=\zeta_{\Delta_{1,u}}(0)\log t+ \zeta'_{\Delta_{1,u}}(0)$, where $\zeta_{\Delta_{t,u}}$
denote the  Zeta function  associated to the data $\bigl((TX,th_{X,u}); (\mathcal{O},h_{\mathcal{O}})\bigr)$. We verify that:
\[\widetilde{Td}(TX,t\,h_{X,u},h_{X,u})=\frac{1}{2}\log t+\frac{1}{6}(\log t)\, c_1(TX,h_{X,u}),\]
in $\oplus_{p\geq 0}\widetilde{A}^{(p,p)}(X)$, see \cite{Character} for the definition of  Bott-Chern classes. Since $\mathrm{Vol}_{th_{X,u}}=t^{\dim X}\mathrm{Vol}_{h_{X,u}}$, then
\[h_{L^2,((TX,th_{X,u}); (\mathcal{O},h_{\mathcal{O}}))}=t^{2\dim H^0(X,\mathcal{O})} h_{L^2,((TX,h_{X,u}); (\mathcal{O},h_{\mathcal{O}}))}. \]
Recall that,
\[
h_{Q,((TX,h_{X,u}); (\mathcal{O},h_{\mathcal{O}}))}=h_{L^2,((TX,th_{X,u}); (\mathcal{O},h_{\mathcal{O}}))}\exp\bigl(-\zeta_{\Delta_{1,u}}'(0)\bigr).
\]
Using the  previous anomaly formula, we get
\[
 \begin{split}
  -2\dim H^0(X,\mathcal{O}) \log t+ \zeta_{\Delta_{1,u}}(0)\log t &=\frac{1}{2}\log t\int_Xc_1(\mathcal{O},h_\mathcal{O})+ \frac{1}{6}\log t \int_Xc_1(TX,h_X).
 \end{split}
\]
Remember that $\zeta_{\Delta_{X,u}}(0)=a_{u,0}$  (see for instance \cite[theorem.1]{Soulé}), hence:
\begin{equation}\label{a_0}
 a_{u,0}=\frac{1}{6}\int_X c_1(TX)+2\quad \forall\, u>1.
\end{equation}
Therefore, $a_{u,0}$ does not   depend on $u$.
\end{proof}

\begin{proof}

In order to study $\theta_{X,\infty}$ we introduce  for any $u\geq 1$, the following auxiliary function:
\[
 \theta_{u,\infty}(t):=\bigl\|P^u e^{-t\Delta_{X,u}}\bigr\|_{1,\infty}\quad \forall \, t>1.
\]
Note that this function is finite for any $t>0$, this is follows easily from \eqref{n5}. We have for any $u,u'>1$,
{\allowdisplaybreaks
\begin{align*}
\Bigl| \theta_{u,\infty}(t)- &\theta_{u',\infty}(t) \Bigr|=\Bigl|\Bigl\| P^u e^{-t\Delta_{X,u}}\Bigr\|_{1,\infty}-\Bigl\|  P^{u'} e^{-t\Delta_{X,u'}}\Bigr\|_{1,\infty}  \Bigr|\\
&\leq \Bigl\| P^u e^{-t\Delta_{X,u}}-P^{u'}e^{-t\Delta_{X,u'}}  \Bigr\|_{1,\infty}\\
&=\biggl\|\int_u^{u'} \frac{\pt}{\pt v}\bigl(P^v e^{-t\Delta_{X,v}}\bigr)dv  \biggr\|_{1,\infty}\\
&=\biggl\|\int_u^{u'}\Bigl( \frac{\pt P^v}{\pt v} e^{-t\Delta_{X,v}}+P^v \frac{\pt }{\pt v}\bigl(\mathcal{O}^{-t\Delta_{X,v}} \bigr)\Bigr)dv  \biggr\|_{1,\infty}\\
&=\biggl\|\int_u^{u'} \Bigl(\frac{\pt P^v}{\pt v} P^ve^{-t\Delta_{X,v}}+P^v \frac{\pt }{\pt v}\bigl(\mathcal{O}^{-t\Delta_{X,v}} \bigr)\Bigr)dv  \biggr\|_{1,\infty}\\
&\leq  \biggl\|\int_u^{u'} \frac{\pt P^v}{\pt v} P^ve^{-t\Delta_{X,v}}dv\biggl\|_{1,\infty}+ \biggl\|\int_u^{u'}P^v\biggl( \frac{\pt }{\pt v}\bigl(\mathcal{O}^{-\frac{t}{2}\Delta_{X,v}}\bigr)e^{-\frac{t}{2}\Delta_{X,v}}+e^{-\frac{t}{2}\Delta_{X,v}}\frac{\pt }{\pt v}\bigl(\mathcal{O}^{-\frac{t}{2}\Delta_{X,v}}\bigr) \biggr)dv  \biggr\|_{1,\infty}\\
&= \biggl\|\int_u^{u'} \frac{\pt P^v}{\pt v} P^ve^{-t\Delta_{X,v}}dv\biggl\|_{1,\infty}+\biggl\|\int_u^{u'} \Bigl(P^v\frac{\pt }{\pt v}\bigl(\mathcal{O}^{-\frac{t}{2}\Delta_{X,v}}\bigr)\Bigr)\Bigl(P^v e^{-\frac{t}{2}\Delta_{X,v}}\Bigr)+\Bigl(P^v e^{-\frac{t}{2}\Delta_{X,v}}\Bigr)\frac{\pt }{\pt v}\bigl(\mathcal{O}^{-\frac{t}{2}\Delta_{X,v}}\bigr) \Bigr)dv  \biggr\|_{1,\infty}\\
&\leq  \int_u^{u'} \biggl\|\frac{\pt P^v}{\pt v} \biggl\|_{L^2,\infty}\theta_{v,\infty}(t)dv+ 2\int_u^{u'} \biggl\|   \frac{\pt }{\pt v}e^{-\frac{t}{2}\Delta_{X,v}}\biggr\|_{L^2,\infty} \theta_{v,\infty}(\frac{t}{2})dv\quad \text{by}\;\eqref{normetrace} \\
&\leq c_7\int_u^{u'}\frac{1}{2^v}\theta_{v,\infty}(t)dv+ c_6\int_u^{u'}\frac{1}{2^v}t^{\frac{1}{4}}\theta_{v,\infty}(\frac{t}{2})dv\\
&\leq c_8\frac{1}{2^u}\int_u^{u'}\Bigl(\theta_{v,\infty}(t)+t^{\frac{1}{4}}\theta_{v,\infty}(\frac{t}{2})\Bigr)dv.
 \end{align*}}
We used the following  facts: $\frac{\pt }{\pt v}\bigl(e^{-t\Delta_{X,v}}\bigr)=\frac{\pt }{\pt v}\bigl(e^{-
t\Delta_{X,v}}\bigr)P^v$
and  $\frac{\pt P^v}{\pt v}P^v=\frac{\pt P^v}{\pt v}Id$. Let us prove the first one: Let $v>1$, and $\xi \in \h$. There exists
$a(\xi)$, a constant, such that $\xi-P^v\xi=a(\xi)$, then:
\[
\frac{e^{-t\Delta_{X,u}}-e^{-t\Delta_{X,u'}}}{u-u'}\Bigl(\xi-P^v\xi\Bigr)=\frac{1}{u-u'}\Bigl( a(\xi)-a(\xi)
\Bigr)=0\quad \forall\, u\neq u',
\]
(since $e^{-t\Delta_{X,u}}\phi=e^{-t\la} \phi$, if $\Delta_{X,u}\phi=\la \phi$). Therefore, for any $s\in \C$ such that $\mathrm{Re}(s)>1$:
\[
\begin{split}
\biggl|\int_0^\infty t^{s-1}\theta_{u,\infty}(t)dt-\int_0^\infty t^{s-1}\theta_{u',\infty}(t)dt\biggr|&\leq c_8\int_u^{u'}\frac{1}{2^v}\int_0^\infty\Bigl(t^{\mathrm{Re}(s)-1} \theta_{v,\infty}(t)+ t^{\mathrm{Re}(s)+\frac{1}{4}-1}\theta_{v,\infty}(\frac{t}{2})\Bigr)dtdv.\\
\end{split}
\]

Now, we consider  the following function $\zeta_{u,\infty}$:
\[
 \zeta_{u,\infty}(s):=\frac{1}{\Gamma(s)}\int_0^\infty \theta_{u,\infty}(t)t^{s-1}dt,\quad \forall s\in \C.
\]

If we let $B=P^u$ in  \eqref{thetainfty}, we get
\begin{equation}\label{ZZZZZ}
\frac{1-\eps}{1+\eps}\theta_{u,\infty}(t)\leq \theta_{X,u}(t)\leq \frac{1+\eps}{1-\eps}\theta_{u,\infty}(t),\quad \forall t>0,
\end{equation}
this is yields to
\begin{equation}\label{zetaeps11}
\frac{1-\eps}{1+\eps}\zeta_{u,\infty}(s)\leq \zeta_{X,u}(s)\leq \frac{1+\eps}{1-\eps}\zeta_{u,\infty}(s),\quad \forall s\in \R.
\end{equation}
Since $\zeta_{X,u}(s)$ is finite for any $s\in \C$ such that $\mathrm{Re}(s)>1$ and  any $u\geq 1$, it follows that $\zeta_{\infty,u}$ is finite for 
any $s\in \C$ such that $\mathrm{Re}(s)>1$. Then, for any $s\in \C$ with $\mathrm{Re}(s)>1$, we have
\[
\begin{split}
\biggl|\zeta_{u,\infty}(s)-\zeta_{u',\infty}(s)\biggr|&\leq \frac{c_8}{2^u} \int_u^{u'}\biggl(\frac{\Gamma\bigl(\mathrm{Re}(s)\bigr) }{\bigl| \Gamma(s)\bigr|}\zeta_{v,\infty}\bigl(\mathrm{Re}(s)\bigr)+2^{\mathrm{Re}(s)+\frac{1}{4}}\frac{\Gamma\bigl(\mathrm{Re}(s)+\frac{1}{4}\bigr)}{\bigl|\Gamma(s)\bigr|}\zeta_{v,\infty}\bigl(\mathrm{Re}(s)+\frac{1}{4}\bigr)\biggr)dv\\
\end{split}
\]

We claim that, it is   possible to  suppose that the first nonzero eigenvalue of
$\Delta_{X,v}$
is $> 1$, for any $v> 1$. Indeed, let $v>1$.  If we multiply  $\omega_{X,v}$ by $t>0$, then the corresponding first nonzero
eigenvalue is $\frac{1}{t}\la_{v,1}$. And, from \eqref{Cheeger} and \eqref{infboundvp}, we have
$\frac{1}{4}(\frac{c_1}{c_2})^2 h_{tg_1}(X)^2\leq \frac{1}{4}h_{g_v}(X)^2\leq \frac{1}{t}\la_{v,1}$ for any $v\geq 1$. Since,
$h_{tg_1}(X)=t^{-\frac{1}{2}}h_{g_1}(X)$, we obtain: $\frac{1}{4t}(\frac{c_1}{c_2})^2 h_{g_1}(X)^2\leq
\frac{1}{t}\la_{v,1}$ for any $v\geq 1$. Hence, for $t$ sufficiently small, we can assume that $\la_{v,1}>1$ for any $v\geq 1$.\\

 Therefore, we get for any  $s>1$:
\begin{align*}
\zeta_{X,v}(s)-\zeta_{X,v}(s+\frac{1}{4})&=\sum_{k=1}^\infty \frac{1}{\la_{v,k}^s}-\sum_{k=1}^\infty \frac{1}{\la_{v,k}^{s+\frac{1}{4}}} \\
&\geq \sum_{k=1}^\infty \frac{1}{\la_{v,k}^s}-\frac{1}{\la_{v,1}^s}\sum_{k=1}^\infty \frac{1}{\la_{v,k}^{s}}\\
&=  \Bigl(1-\frac{1}{\la_{1,v}^{\frac{1}{4}}}\Bigr)\zeta_{X,v}(s).
\end{align*}
Then, using \eqref{zetaeps11},  it is possible to have:
\[
\zeta_{v,\infty}(s)\geq \zeta_{v,\infty}(s+\frac{1}{4}),\quad \forall s>1,\forall v\gg1.
\]
Therefore,
\begin{equation}\label{gronwall}
\begin{split}
\Bigl|\zeta_{u,\infty}(s)-\zeta_{u',\infty}(s)\Bigr|&\leq \frac{c_8}{2^u}\biggl(\frac{\Gamma\bigl(\mathrm{Re}(s)\bigr)}{|\Gamma(s)|}+2^{\mathrm{Re}(s)+\frac{1}{4}}\frac{\Gamma\bigl(\mathrm{Re}(s)+\frac{1}{4}\bigr)}{|\Gamma(s)|} \biggr) \int_u^{u'}\zeta_{v,\infty}\bigl(\mathrm{Re}(s)\bigr)dv,\; \forall\, \mathrm{Re}(s)>1.\\
\end{split}
\end{equation}
Now we assume $s>1$. Let us prove that, for any $ \forall u,u'\gg 1$ and any  $ s>1$:
{\footnotesize{\begin{equation}\label{zetarapport}
\exp\biggl(-\frac{c_8}{\log 2}\Bigl|\frac{1}{2^u}-\frac{1}{2^{u'}} \Bigr|\biggl(\frac{\Gamma\bigl(\mathrm{Re}(s)\bigr)}{|\Gamma(s)|}+2^{s+\frac{1}{4}}\frac{\Gamma(s+\frac{1}{4})}{\Gamma(s)} \biggr) \biggr)\leq \frac{\zeta_{u,\infty}(s)}{\zeta_{u',\infty}(s)}\leq \exp\biggl(\frac{c_8}{\log 2}\Bigl|\frac{1}{2^u}-\frac{1}{2^{u'}} \Bigr|\biggl(\frac{\Gamma\bigl(\mathrm{Re}(s)\bigr)}{|\Gamma(s)|}+2^{s+\frac{1}{4}}\frac{\Gamma(s+\frac{1}{4})}{\Gamma(s)} \biggr) \biggr).
\end{equation}}}
If $u'\mapsto \zeta_{u',\infty}(s)$ is a derivable function  for $s$ fixed, we obtain using \eqref{gronwall}
\[
\begin{split}
\biggl|\frac{\pt }{\pt u}\zeta_{u,\infty}(s)\biggr|&\leq \frac{c_8}{2^u}\biggl(\frac{\Gamma\bigl(\mathrm{Re}(s)\bigr)}{|\Gamma(s)|}+2^{\mathrm{Re}(s)+\frac{1}{4}}\frac{\Gamma\bigl(\mathrm{Re}(s)+\frac{1}{4}\bigr)}{|\Gamma(s)|} \biggr)\zeta_{u,\infty}(s),\\
\end{split}
\]
then,
\[
\biggl|\frac{\pt }{\pt u}\log \zeta_{u,\infty}(s)\biggr|\leq
\frac{c_8}{2^u}\biggl(\frac{\Gamma\bigl(\mathrm{Re}(s)\bigr)}{|\Gamma(s)|}+2^{s+\frac{1}{4}}\frac{\Gamma(s+\frac{1}{4}
)}{\Gamma(s)} \biggr),
\]
which gives:
\[
\Bigl|\log \zeta_{u,\infty}(s)-\log \zeta_{u',\infty}(s)\Bigr|\leq \frac{c_8}{\log
2}\Bigl|\frac{1}{2^u}-\frac{1}{2^{u'}}
\Bigr|\biggl(\frac{\Gamma\bigl(\mathrm{Re}(s)\bigr)}{|\Gamma(s)|}+2^{s+\frac{1}{4}}\frac{\Gamma(s+\frac{1}{4})}{\Gamma
(s)} \biggr) \quad \forall u,u'.
\]
On the other hand, if $u'\mapsto \zeta_{u',\infty}(s)$ is not derivable, we apply the  Gronwall lemma to the  function
$u'\mapsto \zeta_{u',\infty}(s)$. Note that this  function is continuous, because it is locally  Lipschitz, which
is a consequence of  the continuity of $v\mapsto \zeta_{X,v}(s)$, and the inequalities  \eqref{zetaeps11} and  \eqref{gronwall}.\\

From \eqref{zetaeps11} and \eqref{zetarapport}, there exist positives constants  $c_{12}$ and $c_{13}$ such that
{\footnotesize{
\begin{equation}\label{zetaboundunif}
c_{12}\exp\biggl(-\frac{c_8}{\log 2}\Bigl|\frac{1}{2^u}-\frac{1}{2^{u'}} \Bigr|\biggl(\frac{\Gamma\bigl(\mathrm{Re}(s)\bigr)}{|\Gamma(s)|}+2^{s+\frac{1}{4}}\frac{\Gamma(s+\frac{1}{4})}{\Gamma(s)} \biggr) \biggr)\leq \frac{\zeta_{X,u}(s)}{\zeta_{X,u'}(s)}\leq c_{13}\exp\biggl(\frac{c_8}{\log 2}\Bigl|\frac{1}{2^u}-\frac{1}{2^{u'}} \Bigr|\biggl(\frac{\Gamma\bigl(\mathrm{Re}(s)\bigr)}{|\Gamma(s)|}+2^{s+\frac{1}{4}}\frac{\Gamma(s+\frac{1}{4})}{\Gamma(s)} \biggr) \biggr)\;
\end{equation}}}
for any $u,u'\gg1$ et $ \forall s>1$. \\

We know that $\zeta_{X,u}$ is  holomorphic on the open set  $\{s\in \C|\mathrm{Re}(s)>1\}$. Let us prove that
$\bigl(\zeta_{X,u}\bigr)_{u\geq 1}$ converges  to a holomorphic function defined on $\{s\in \C|\mathrm{Re}(s)>1\}$ uniformly on any
domain of the  form $\al\leq \mathrm{Re}(s)\leq \beta$, where
$1<\al\leq \beta$. We will show that this function is in fact  $\zeta_{X,\infty}$. \\

From \eqref{ZZZZZ}, we have
\begin{align*}
 -\frac{2\eps}{1+\eps}\theta_{\infty,u}(t)\leq \theta_{X,u}(t)-\theta_{\infty,u}(t)\leq \frac{2\eps}{1-\eps}\theta_{u,\infty}\quad \forall \,t>0\;\forall u,u'\gg 1
\end{align*}
Then,
\begin{align*}
 \Bigl|\theta_{X,u}(t)-\theta_{\infty,u}(t) \Bigr|\leq \frac{2\eps}{1-\eps}\theta_{u,\infty}(t),
\end{align*}
It follows that,
\begin{align*}
 \Bigl| \zeta_{X,u}(s)-\zeta_{u,\infty}(s)\Bigr|\leq \frac{2\eps}{1-\eps}\frac{\Gamma(\mathrm{Re}(s))}{|\Gamma(s)|}\zeta_{u,\infty}(\mathrm{Re}(s))\quad \forall\, \mathrm{Re}(s)>1.
\end{align*}
Using this inequality, we get:
\begin{align*}
 \Bigl| \zeta_{X,u}(s)-\zeta_{X,u'}(s) \Bigr|&\leq   \Bigl| \zeta_{u,\infty}(s)-\zeta_{u',\infty}(s)\Bigr|+\frac{2\eps}{1-\eps}\frac{\Gamma(\mathrm{Re}(s))}{|\Gamma(s)|}\Bigl(\zeta_{u,\infty}(\mathrm{Re}(s))+\zeta_{u',\infty}\bigl(\mathrm{Re}(s)\bigr)\Bigr),
 \end{align*}
for any $s\in \C$ such that $\mathrm{Re}(s)>1$ and for any $u,u'\gg 1$.\\

If we choose $1<\al<\beta$.  Then,  we have already proved that  $\bigl(\zeta_{u,\infty}(\mathrm{Re}(s))\bigr)_u$ is uniformly
bounded with respect to $u$ on $\al\leq \mathrm{Re}(s)\leq \beta$. Then, we can find a constant $K$, which depend uniquely on
  $\al$ and $\beta$ such that
  \begin{align*}
 \Bigl| \zeta_{X,u}(s)-\zeta_{X,u'}(s) \Bigr|&\leq   \Bigl| \zeta_{u,\infty}(s)-\zeta_{u',\infty}(s)\Bigr|+\frac{2\eps}{1-\eps}K,
\end{align*}
for any $s\in \C$ such that $\mathrm{Re}(s)>1$ and for any $u,u'\gg 1$.

Therefore, $\bigl(\zeta_{X,u}\bigr)_{u\geq 1}$ converges uniformly to a limit on the set $\{s\in \C\,|\,\al\leq\mathrm{Re}(s)\leq
\beta\}$, and this limit is necessarily holomorphic on $\{s\in \C\,|\,\al\leq\mathrm{Re}(s)\leq
\beta\}$. \\

The following claim will be used to establish that $(\zeta_{X,u})_{u\geq 1}$ converges pointwise to $\zeta_{X,\infty}$
on the set $\{s\in \C\,|\, \mathrm{Re}(s)>1 \}$.
\begin{claim}
Let $\theta$ be  a positive decreasing  function on $\R^+$. We let  $\zeta$ be  the  function given by:
\[
 \zeta(s)=\frac{1}{\Gamma(s)}\int_0^\infty t^{s-1}\theta(t)dt,
\]
for  $s\in \R$.
We have,
\begin{equation}\label{thetazeta1111}
 \theta(a)\leq \frac{\Gamma(s+1)}{a^{s}}\zeta(s),\quad \, \forall s>a>0.
\end{equation}

\end{claim}
\begin{proof}
Let $a>0$  and $s>a$, we have
\[
 \begin{split}
  \zeta(s)&=\frac{1}{\Gamma(s)}\int_0^\infty \theta(t)t^{s-1}dt\\
&=\frac{1}{\Gamma(s)}\int_0^a \theta(t)t^{s-1}dt+\frac{1}{\Gamma(s)}\int_a^\infty \theta(t)t^{s-1}dt\\
&\geq \frac{\theta(a)}{\Gamma(s)}\int_0^a t^{s-1}dt\\
&=\frac{\theta(a)}{\Gamma(s+1)}a^s.
 \end{split}
\]
\end{proof}
We will apply this  lemma in order to show that $\theta_{X,u}(t)$ uniformly  bounded with respect to $u$ for any $t>0$. From \eqref{thetazeta1111},
we have,
\[
\theta_{X,u}(t)\leq \frac{\Gamma(s+1)}{t^{s}}\zeta_{X,u}(s),\quad \, \forall s>t>0.
\]
then, if we choose $s>1$ so $\zeta_{X,u}(s)$ is finite. Using \eqref{zetaboundunif}, we deduce that for any fixed $t>0$, $\theta_{X,u}(t)$ is uniformly bounded with respect to $u$. \\

From \eqref{thetainfty}, \eqref{thetainfty11} and \eqref{thetainfty12} we get
\[
 \bigl(\theta_{X,u}(t)\bigr)_u\xrightarrow[u\mapsto \infty]{} \theta_{X,\infty}(t)\quad \forall t>0.
\]
And using \eqref{a_uuu}, we deduce easily that:
\[
\bigl(\rho_{X,u}(t)\bigr)_{u} \xrightarrow[u\mapsto \infty]{} \rho_{X,\infty}(t)\quad \forall t>0.
\]

Fix $\eps>0$, since $\bigl(\zeta_{X,u}(1+\eps) \bigr)_{u\geq 1}$ is a convergent sequence, we can find a constant $c$ such that:
\[
 \theta_{X,u}(t)\leq \frac{1}{t^{\eps +1}}c,\quad\forall\, 0<t<1+\eps,\; \forall u\gg1.
\]
then,
\[
 \theta_{X,\infty}(t)\leq \frac{1}{t^{\eps+1}}c.
\]
Let $s\in \C$ such that $\mathrm{Re}(s)>1+\eps$. We have
\[
 \zeta_{X,u}(s)=\frac{1}{\Gamma(s)}\int_0^\delta t^{s-\eps-2} \bigl(t^{\eps+1}\theta_{X,u}(t)  \bigr)dt+\frac{1}{\Gamma(s)}\int_\delta^\infty t^{s-1}\theta_{X,u}(t)dt,\; \forall u\geq 1.
\]
Since $\bigl(\theta_{X,u}\bigr)_u$ converges pointwise to $\theta_\infty$, and by the dominated convergence theorem, we get:
\[
 (\zeta_{X,u}(s))_{u\geq 1}\xrightarrow[u\mapsto \infty]{} \zeta_{X,\infty}(s)\quad \forall\, s\in \C\, s.t. \,\mathrm{Re}(s)>1+\eps.
\]

Let us prove the following
\[
 \zeta_{X,\infty}(s)=\sum_{k=1}^\infty \frac{1}{\la_{\infty,k}^s},
\]
for any $s\in \C$ such that $\mathrm{Re}(s)>1$. Let us show first
\[
 \zeta_{X,\infty}(s)=\sum_{k=1}^\infty \frac{1}{\la_{\infty,k}^s},\; \forall\, s>1.
\]

Let $\delta>0$, we have for any $\mathrm{s}> 1+\eps$:
\begin{align*}
 \zeta_{X,\infty}(s)&=\frac{1}{\Gamma(s)}\int_0^\delta t^{s-1}\theta_{X,\infty}(t)dt+\int_\delta^\infty t^{s-1}\theta_{X,\infty}(t)dt\\
&=\frac{1}{\Gamma(s)}\int_0^\delta \Bigl(\theta_{X,\infty}(t) t^{\eps+1} \Bigr)t^{s-2-\eps}dt+\int_\delta^\infty t^{s-1}\theta_{X,\infty}(t)dt.
\end{align*}
Since $\Bigl|\frac{1}{\Gamma(s)}\int_0^\delta \Bigl(\theta_{X,\infty}(t) t^{\eps+1} \Bigr)t^{s-2-\eps}dt\Bigr|\leq
\frac{c}{s-1-\eps}\delta^{s-1-\eps}$ and $\theta_{X,\infty} (t)\leq \theta_{X,\infty}(\delta)e^{-\la_{\infty,1}(t-\delta)}$ for any $t\geq
\delta$, then
\begin{align*}
 \zeta_{X,\infty}(s)&=O(\delta^{s-1-\eps})+\sum_{k=1}^\infty \frac{1}{\Gamma(s)}\int_\delta^\infty t^{s-1}e^{-\la_{\infty,k} t}dt\\
&\leq  O(\delta^{s-1-\eps})+\sum_{k=1}^\infty \frac{1}{\la_{\infty,k}^s}\frac{1}{\Gamma(s)}\int_{\la_{\infty,k}\delta}^\infty t^{s-1}e^{- t}dt\\
&\leq O(\delta^{s-1-\eps})+\sum_{k=1}^\infty \frac{1}{\la_{\infty,k}^s}.
\end{align*}

Remark that
\[
 \sum_{k=1}^\infty \frac{1}{\la_{\infty,k}^s}=\lim_{N\mapsto \infty}\sum_{k=1}^N \frac{1}{\la_{\infty,k}^s}\leq \frac{1}{\Gamma(s)}\int_0^\infty t^{s-1}\theta_{X,\infty}(s)dt=\zeta_{X,\infty}(s).
\]

We conclude that
\[
 \zeta_{X,\infty}(s)=\sum_{k=1}^\infty \frac{1}{\la_{\infty,k}^s},\; \forall s>1.
\]
Let $s\in \C$ such that $\mathrm{Re}(s)>1$. For any $N\in \N$, we set
\[
 \zeta_{N,\infty}(s):=\frac{1}{\Gamma(s)}\int_0^\infty t^{s-1}\Bigl(\sum_{k=N}^\infty e^{-\la_{\infty,k}t}\Bigr)dt,
\]
then
\[
\zeta_{N,\infty}(s)=\zeta_{X,\infty}(s)-\sum_{k=1}^{N-1}\frac{1}{\la_{\infty,k}^s},\quad \forall\, \mathrm{Re}(s)>1.
\]
We have,
\[
\bigl|\zeta_{N,\infty}(s)\bigr|\leq \Bigl|\frac{1}{\Gamma(s)}\Bigr|\int_0^\infty t^{\mathrm{Re}(s)-1}\Bigl(\sum_{k=N}^\infty e^{-\la_{\infty,k}t}\Bigr)dt=\frac{\Gamma(\mathrm{Re}(s))}{\bigl| \Gamma(s)\bigr|}\Bigl(\zeta_{X,\infty}(\mathrm{Re}(s))-\sum_{k=1}^{N-1}\frac{1}{\la_{\infty,k}^{\mathrm{Re}(s)}}  \Bigr).
\]
The right-hand side converges to  zero when $N$ goes to infinity. We conclude that
\[
 \zeta_{X,\infty}(s)=\sum_{k=1}^\infty \frac{1}{\la_{\infty,k}^{s}},\; \forall \,s\in \C\,s.t. \,\mathrm{Re}(s)>1.
\]

 Let $s\in \C$ such that $\mathrm{Re}(s)>1$, we have
\begin{align*}
 \bigl|\zeta_{X,u}(s)-\zeta_{X,u'}(s)\bigr|&\leq \bigl| \zeta_{X,u}(s)\bigr|+ \bigl| \zeta_{X,u'}(s)\bigr|\\
&\leq \zeta_{X,u}(\mathrm{Re}(s))+\zeta_{X,u'}(\mathrm{Re}(s)).
\end{align*}
(because $\zeta_{X,u}(s)=\sum_{k=1}^\infty \frac{1}{\la_{u,k}^s}$ when  $s\in \C$ such that $\mathrm{Re}(s)>1$). \\

Let $x\in \C$ such that $\mathrm{Re}(x)>0$. For any $u\geq 1$, we set:
\begin{align*}
\widetilde{\theta}_u(x)=\frac{1}{2\pi  i}\int_{c-i\infty}^{c+i\infty}x^{-s}\Gamma(s)\zeta_{X,u}(s)ds,
\end{align*}
where $c$ a fixed integer  greater that $1$.\\

One verify the following assertions: $\widetilde{\theta}_u(x)=\sum_{k\geq 1} e^{-\la_{u,k}x}$, $\bigl|\widetilde{\theta}_u(x) \bigr|\leq \theta(\mathrm{Re}(x))$, $\widetilde{\theta}_u$ and $\theta_{X,u}$ are equal on $\R^{+\ast}$, and that $\widetilde{\theta}_u(x)=\frac{a_{-1}}{x}+a_0+\widetilde{\rho}(x)$ for  $x$ small enough.\\

We have,
{\allowdisplaybreaks
\begin{align*}
 \biggl|\widetilde{\theta}_u(x)-\widetilde{\theta}_{u'}(x) \biggr|&\leq \frac{1}{2\pi i}\int_{c-i\infty}^{c+i\infty}\Biggl|\Bigl(x^{-s}\Gamma(s)\zeta_{X,u}(s)-x^{-s}\Gamma(s)\zeta_{X,u'}(s)\Bigr) \Biggr|ds\\
&\leq \frac{1}{2\pi  i}\int_{c-i\infty}^{c+i\infty }|x|^{\mathrm{Re}(s)}\bigl|\Gamma(s)\bigr|\Bigl(\zeta_{X,u}\bigl(\mathrm{Re}(s)\bigr)+\zeta_{X,u'}\bigl(\mathrm{Re}(s)\bigr) \Bigr)ds \\
&=\frac{1}{2\pi  }|x|^{-c}\Bigl(\zeta_{X,u}(c)+\zeta_{X,u'}(c)\Bigr)\biggl(\int_{-1}^{1 }\bigl|\Gamma(c+it)\bigr| dt+\int_{]-\infty,-1]\cup[1,\infty[}\bigl|\Gamma(c+it)\bigr| dt \biggr)\\
&=\frac{1}{2\pi  }|x|^{-c}\Bigl(\zeta_{X,u}(c)+\zeta_{X,u'}(c)\Bigr)\biggl(\int_{-1}^1 \bigl|\Gamma(c+it) \bigr|dt+\int_{]-\infty,-1]\cup[1,\infty[}\prod_{k=0}^{c-1}\sqrt{k^2+t^2}\bigl|\Gamma(it)\bigr|dt\biggr)\\
&=\frac{1}{2\pi  }|x|^{-c}\Bigl(\zeta_{X,u}(c)+\zeta_{X,u'}(c)\Bigr)\biggl(\int_{-1}^1 \bigl|\Gamma(c+it) \bigr|dt+\int_{]-\infty,-1]\cup[1,\infty[}\prod_{k=0}^{c-1}\sqrt{k^2+t^2}\frac{\sqrt{\pi}}{\sqrt{|t|}\sqrt{|\sinh(\pi t)|}}\bigr| dt\biggr)\\
&\quad \text{see formula}\;\cite[6.1.29]{Table2}.
\end{align*}}
We verify that the last  integral is convergent.  Since $\bigl(\zeta_{X,u}\bigr)_{u\geq 1}$ converges uniformly on any  domain of the  form $\{s\in \C\,|\, \mathrm{Re}(s)\geq \delta+1\}$. Then, there exists constant  $K$ which depend only on $c$, such that:
\[
 \bigl|\widetilde{\theta}_u(x)-\widetilde{\theta}_{u'}(x) \bigr|\leq K |x|^{-c}, \quad\forall \,x\in \C\,s.t.\, \mathrm{Re}(x)>0.
\]
We conclude that
\[
\bigl|\widetilde{\rho}_u(x)-\widetilde{\rho}_{u'}(x) \bigr|\leq K |x|^{-c}+\bigl|a_{u,-1}-a_{u',-1} \bigr||x|^{-1}, \quad\forall \,x\in \C\,s.t.\, \mathrm{Re}(x)>0.
\]
Let $r>0$ fixed. We denote by $D$ the curve in the complex plan, given by  $r e^{i\al} $, such that $-\frac{\pi}{2}\leq\al \leq
\frac{\pi}{2}$.

  If we  replace $x$ by $x^2$ in the above inequality, and we consider  $k$, an integer greater that $1$, then we get
\begin{align*}
 \biggl|\int_D\frac{\widetilde{\rho}_u(x^2)-\widetilde{\rho}_{u'}(x^2)}{x^{2k}}dx \biggr|\leq K \int_D
 |x|^{-2c-2k}dx+\bigl|a_{u,-1}-a_{u',-1} \bigr|\int_D |x|^{-2-2k}dx,
\end{align*}
this is yields to
\begin{align*}
 \bigl|a_{u,k}-a_{u',k} \bigr|\leq K  r^{-2c-2k}+\bigl|a_{u,-1}-a_{u',-1} \bigr|r^{-2-2k}\quad\forall\, u,u'\; \forall \,k\in \N_{\geq 1},
\end{align*}
(we have used the following fact: $\int_D x^{2(j-k)}dx=\delta_{k,j}\pi$, and recall that $a_{u,0}$ does not depend on $u$).

Therefore, if we take $0<t<r$, we get
{\allowdisplaybreaks
\begin{align*}
 \bigl|\rho_{X,u}(t^2) \bigr|&\leq \sum_{k\geq 1} \bigl|a_{u,k}\bigr|t^{2k}\\
&\leq \sum_{k\geq 1}\bigl| a_{u,k}-a_{u',k}\bigr|t^{2k}+\sum_{k=1}^\infty |a_{u',k}|t^{2k}\\
&\leq \sum_{k\geq 1}Kr^{-2c} \Bigl(\frac{t^2}{r^2}\Bigr)^{k}+ \sum_{k\geq 1}r^{-2} \Bigl(\frac{t^2}{r^2}\Bigr)^{k}\bigl|a_{u,-1}-a_{u',-1} \bigr|+\sum_{k=1}^\infty |a_{u',k}|t^{2k}\\
&\leq K r^{-2c}\frac{t^2}{r^2-t^2}+ r^{-2}\frac{t^2}{r^2-t^2}\bigl|a_{u,-1}-a_{u',-1} \bigr|+\sum_{k=1}^\infty |a_{u',k}|t^{2k}.
\end{align*}}

Now, recall that $(a_{u,-1})_{u\geq 1}$ is bounded, see \eqref{a_uuu}. If we fix $u'$, we can find a constant $K'$ and a real  $t_0>0$ such that
\[
 \sum_{k=1}^\infty |a_{u',k}|t^{2k}\leq K'\frac{t^2}{1-t^2}\quad \forall \, t\in[0,t_0 ].
\]
We conclude there exists a  constant $K''$ such that:
\begin{align*}
 \bigl|\rho_{X,u}(t) \bigr|\leq K'' t,\quad \forall\, u\gg 1,\; \forall\, 0\leq t\leq \min(\sqrt{r},t_0).
\end{align*}
But, we proved that  $\rho_u$ converges pointwise to $\rho_\infty$, then
\[
 \bigl|\rho_{X,\infty}(t)\bigr|\leq K'' t,\quad \forall\,  0\leq t\leq \min(\sqrt{r},t_0).
\]
that is $\rho_{X,\infty}(t)=O(t)$. As a first consequence, $\zeta_{X,\infty} $ admits a holomorphic continuation at   $s=0$, and
\[
 \zeta'_{X,\infty}(0)=\int_1^\infty \frac{\theta_{X,\infty}(t)}{t}dt+a_{\infty,-1}+a_{\infty,0}+\int_0^1\frac{\rho_{X,\infty}(t)}{t}dt.
\]
Since,
\[
 \theta_{X,u}(t)\leq \theta_{X,u}(1)e^{-\la_{u,1}(t-1)}\quad \forall t\geq 1.
\] and from  \eqref{lowerbound}, we can find a constant $\kappa>0$ such that $\la_{u,1}\geq\kappa$ for any $u\geq 1$. Recall that $\bigl(\theta_{X,u}(1)\bigr)_{u}$ is bounded. Therefore, there exists $M>0$ such that
 \[
 \theta_{X,u}(t)\leq M e^{-\kappa t}\quad \forall t\geq 1.
\]
Also, we proved that
\[
 \bigl|\rho_{X,u}(t)\bigr|\leq K'' t,\quad \forall u\gg 1\;\forall\,  0\leq t\leq \min(\sqrt{r},t_0).
\]

We deduce,  using the dominated  convergence theorem, that
\[
 \biggl(\int_1^\infty\frac{\theta_{X,u}(t)}{t}dt\biggr)_u\xrightarrow[u\mapsto \infty]{}\int_1^\infty \frac{\theta_{X,\infty}(t)}{t}dt,
\]
and
\[
\biggl(\int_0^1 \frac{\rho_{X,u}(t)}{t}dt\biggr)_{u}\xrightarrow[u\mapsto \infty]{} \int_0^1\frac{\rho_{X,\infty}(t)}{t}dt.
\]
Then,
\[
 \bigl( \zeta'_{X,u}(0)\bigr)_{u\geq 1}\xrightarrow[u\mapsto \infty]{}\zeta'_{X,\infty}(0).
\]
In particular,
\begin{equation}\label{lllllll}
 \bigl( \zeta'_{X,p}(0)\bigr)_{p\in \N }\xrightarrow[p\mapsto \infty]{}\zeta'_{X,\infty}(0).
\end{equation}

\end{proof}

\begin{theorem}\label{convergenceAnomaly}
We keep the same assumptions. For any $p\in \N$, let $ h_{Q,((X,\omega_{X,p});(\mathcal{O},h_{\mathcal{O}})
)}$ be the Quillen metric associated to $\bigl((X,\omega_{X,p});(\mathcal{O},h_{\mathcal{O}}) \bigr)$. We have, the sequence $\Bigl(
h_{Q,((X,\omega_{X,p});(\mathcal{O},h_{\mathcal{O}}) )}\Bigr)_{p\in \N}$ converges to a limit, which does not depend on the choice of
$\bigl(h_{X,p}\bigr)_{p\in \N}$.  We denote this limit by $h_{Q,((X,\omega_{X,\infty});(\mathcal{O},h_{\mathcal{O}})
)}$.
\end{theorem}
\begin{proof}
By the anomaly formula, see \cite{BGS1}, we have for any $p,q\in \N$:
\[
\begin{split}
 -\log h_{Q,((TX,th_{X,p}); (\mathcal{O},h_{\mathcal{O}}))}+\log h_{Q,((TX,h_{X,q}); (\mathcal{O},h_{\mathcal{O}}))}&=\int_X
 ch(\mathcal{O},h_{\mathcal{O}})\widetilde{\mathrm{Td}}(TX,h_{X,p},h_{X,q}).
\end{split}
\]
So, it suffices to prove that the right hand side forms a Cauchy sequence. We have:

\[
\widetilde{\mathrm{Td}}\bigl(TX,h_{X,p},h_{X,q}\bigr)=-\frac{1}{2}\log\biggl( \frac{h_{X,p}}{h_{X,q}}\biggr)-\frac{1}{12}\log\biggl(\frac{h_{X,p}}{h_{X,q}}\biggr)\Bigl(c_1(TX,h_{X,p})+c_1(TX,h_{X,q}) \Bigr).
\]
in $\oplus_{j}\widetilde{A}^{j,j}(X)$, that is the algebra of $(\ast,\ast)$-differential forms on $X$,  modulo $\mathrm{Im}\, \overline{\partial}+\mathrm{Im}\,\partial$, see \cite{character2}.

Then,
\[
\int_X
 ch(\mathcal{O},h_{\mathcal{O}})\widetilde{\mathrm{Td}}(TX,h_{X,p},h_{X,q})=-\frac{1}{12}\int_X
 \log\biggl(\frac{h_{X,p}}{h_{X,q}}\biggr)\Bigl(c_1(TX,h_{X,p})+c_1(TX,h_{X,q}) \Bigr)
\]
Since, for any $n\in \N$, there exist $(h_{n,1})_{n\in \N}$ and $(h_{n,2})_{n\in \N}$ two smooth positive metrics such that
$h_{X,n}=h_{n,1}\otimes h_{n,2}^{-1}$ for any $n\in \N$. Then, we can find a constant $K$, such that:
\[
\Bigl| \int_X
 ch(\mathcal{O},h_{\mathcal{O}})\widetilde{\mathrm{Td}}(TX,h_{X,p},h_{X,q}) \Bigr|\leq  K\biggl\|\log \bigl(\frac{h_{X,p}}{h_{X,q}}\bigr)\biggr\|_{\sup}.
\]
It follows that $\Bigl(
h_{Q,((X,\omega_{X,p});(\mathcal{O},h_{\mathcal{O}}) )}\Bigr)_{p\in \N}$ admits a finite limit. From the previous  inequality, it is clear that the limit
does not depend on the choice of $(h_{X,p})_{p\in \N}$.
\end{proof}

The last theorem extend the notion of Quillen metrics to integrable metrics on compact riemannian surfaces, by using an
approximation process. However,  theorem \eqref{key2} provides a direct way to define the holomorphic analytic torsion in this
situation. It is important to compare both methods, this is given in the following theorem:
\begin{theorem}\label{compare2methods}
We have,
\[
h_{Q,((X,\omega_{X,\infty});(\mathcal{O},h_{\mathcal{O}}) )}=h_{L^2,((X,\omega_{X,\infty});(\mathcal{O},h_{\mathcal{O}}) )}\exp\bigl(\zeta_{X,\infty}'(0)\bigr).
\]
\end{theorem}
\begin{proof}
We have, for any $p\in \N_{\geq 1}$:
\[
h_{Q,((X,\omega_{X,p});(\mathcal{O},h_{\mathcal{O}}))}=h_{L^2,((X,\omega_{X,p});(\mathcal{O},h_{\mathcal{O}}))}\exp\bigl({\zeta_{X,p}'(0)}\bigr).
\]
According to \eqref{lllllll},  it suffices  to prove that
$\bigl( h_{L^2,((X,\omega_{X,p});(\mathcal{O},h_{\mathcal{O}}) )}\bigr)_{p\in \N}$ converges to $h_{L^2,((X,\omega_{X,\infty});(\mathcal{O},h_{\mathcal{O}}) )}$. Indeed, we
have
$ h_{L^2,((X,\omega_{X,p});(\mathcal{O},h_{\mathcal{O}}) )}=\mathrm{Vol}_{L^2,p}\bigl( H^0(X,\mathcal{O})\bigr) \mathrm{Vol}_{L^2,p}\bigl( H ^1(X,\mathcal{O})\bigr)^{-1} $,
for any $p\in \N$.
One verify easily the  convergence of $\bigl(\mathrm{Vol}_{L^2,p}\bigl( H^0(X,\mathcal{O})\bigr)\bigr)_{p\in \N}$ to
$\mathrm{Vol}_{L^2,\infty}\bigl( H^0(X,\mathcal{O})\bigr)$. The convergence of $\bigl(\mathrm{Vol}_{L^2,p}\bigl( H ^1(X,\mathcal{O})\bigr)\bigr)_{p\in
\N}$ follows from \eqref{vol1H}.
\end{proof}
We need the following lemma, in order to prove \eqref{vol1H}:
\begin{lemma}\label{kerH1}
Let $\bigl(X,\omega_X)$ be a compact riemannian surface. Let $\overline{\mathcal{O}}=\bigl(\mathcal{O},h_{\mathcal{O}}\bigr)$ a holomorphic line bundle endowed
with a constant hermitian metric $h_{\mathcal{O}}$. We let $K_X=\Omega_X^{(1,0)}$ which we equip with  the induced metric from $\omega_X$.  We have:
\[
 \Delta^0_{{}_{K_X}}\ast_{1}=\ast_{1}\Delta^1_{\mathcal{O}}.
\]
where $\Delta_{\bullet}^\ast$ is the Laplacian acting on $A^{(0,\ast)}\bigl(X,\bullet\bigr)$. In particular,
\[
 \ker\bigl(\Delta^1_\mathcal{O}\bigr)=\ast^{-1}_{1}\Bigl(H^0\bigl(X,K_X\bigr) \Bigr).
\]
\end{lemma}
\begin{proof}
Let $\xi\in A^{(0,1)}\bigl(X)$. We have (see paragraph \eqref{rappelLAPCLA} for the notations):
\begin{align*}
\Delta_{{}_{K_X}}\bigl(\ast_{1}\xi \bigr)&=\overline{\pt}_{{}_{K_X}}^\ast \overline{\pt}_{{}_{K_X}}\bigl( \ast_{1}\xi\bigr)\\
&=\ast_{1}\overline{\partial}_{\mathcal{O}}\ast_{0}^{-1}\overline{\pt}_{K_X} \ast_{1}\xi\\
&=\ast_{1}\overline{\partial}_{\mathcal{O}}\overline{\partial}_{\mathcal{O}}^\ast \xi\\
&=\ast_{1}\Delta^1_{\mathcal{O}}\xi.
\end{align*}
Since $\ker \Delta^0_{K_X}=H^0\bigl(X,K_X \bigr)$, then
\[
\ker \Delta^1_{\mathcal{O}}=\ast_{1}^{-1}\Bigl(H^0\bigl(X,K_X\bigr) \Bigr).
\]
\end{proof}
\begin{lemma}\label{vol1H}
Let $X$ compact riemannian surface. Let $\bigl(h_p\bigr)_{p\in \N}$ be a sequence of continuous hermitian metrics which converges uniformly to $h_{X,\infty}$ on  $TX$.

 We have
 \[
 \Bigl(\mathrm{Vol}_{L^2,p}\bigl( H ^1(X,\mathcal{O})\bigr) \Bigr)_{p\in \N}\xrightarrow[p\mapsto \infty]{}\mathrm{Vol}_{L^2,\infty}\bigl( H ^1(X,\mathcal{O})\bigr).
 \]
\end{lemma}
\begin{proof}
From \eqref{starhodge}, we see that $\ast_{1}$ does not depend on the metric of $X$.
Let $\xi,\xi'\in A^{(0,1)}\bigl(X,K_X\bigr)$. Since $X$ is compact, we may assume that $\xi=g \,dz$
and $\xi'=fdz$, where
$g,f\in A^{(0,0)}\bigl(X\bigr)$ and $z$ is a local coordinate. We have for any $p,q\in \N$
\[
\ast_{1,p}^{-1}\bigl(g \,dz\bigr)=\ast_{1,q}^{-1}\bigl(g\, dz\bigr),
\]
and,
\begin{align*}
\bigl(\ast_{1,p}^{-1}\xi,\ast_{1,p}^{-1}\xi' \bigr)_{L^2_{X,q}}&=\frac{i}{2\pi}\int_X \overline{g}f\,dz\wedge d\z.
\end{align*}
Therefore, obviously we have
\[
\Bigl(\bigl(\ast_{1,p}^{-1}\xi,\ast_{1,p}^{-1}\xi' \bigr)_{L^2_{X,p}}\Bigr)_{p\in \N}\xrightarrow[p\mapsto \infty]{}\bigl(\ast_{1,\infty}^{-1}\xi,\ast_{1,\infty}^{-1}\xi' \bigr)_{L^2_{X,\infty}}.
\]
By  proposition \eqref{kerH1}, we have \[\ker
\Delta_{X,p}^1=\ast_{1}^{-1}\Bigl(H^0(X,K_X) \Bigr)\quad \forall\, p\in \N_{\geq 1}.\]
Then,
 \[
 \Bigl(\mathrm{Vol}_{L^2_{X,p}}\bigl( H ^1(X,\mathcal{O})\bigr) \Bigr)_{p\in \N}\xrightarrow[p\mapsto \infty]{}\mathrm{Vol}_{L^2_{X,\infty}}\bigl( H ^1(X,\mathcal{O})\bigr).
 \]

\end{proof}

We conclude that:
\[
\Bigl( h_{L^2,((X,\omega_{X,p});(\mathcal{O},h_{\mathcal{O}}) )}\Bigr)_{p\in \N} \xrightarrow[p\mapsto \infty]{} h_{L^2,((X,\omega_{X,\infty});(\mathcal{O},h_{\mathcal{O}}) )}.
\]

\section{Appendix 1}\label{rappelclassic}

\subsection{Integrable metrics}\label{rappelmetint}
Let $X$ be a complex analytic manifold and $\overline{L}=(L,\vc)$ holomorphic line bundle equipped with a continuous hermitian
metric on $L$.
\begin{definition}
\rm{We call first   Chern current of $\overline{L}$, and we denote it by $c_1\bigl( \overline{L}\bigr)\in D^{(1,1)}(X)$, the
$(1,1)$-current defined locally by the following formula:
\[
c_1\bigl(\overline{L}\bigr)=c_1(L,\vc)=dd^c\bigl( -\log \|s\|^2\bigr),
\]
where $s$ is  a nonzero holomorphic  local section of  $L$.}
\end{definition}

\begin{definition}
\rm{The metric $\vc$ is positive if $c_1\bigl(L,\vc\bigr)\geq 0$.}
\end{definition}
\begin{definition}\label{admissibledefinition}
\rm{We say  $\vc$ is  admissible if there exists a sequence  $\bigl(\vc_n \bigr)_{n\in \N}$ of smooth positive hermitian metrics
 converging uniformly to  $\vc$ on $L$. An admissible line bundle on  $X$ is a holomorphic line bundle equipped with
an admissible metric.}
\end{definition}
We say $\overline{L}$ is an integrable line bundle if there exist $\overline{L}_1$ et $\overline{L}_2$ two  admissible line
bundles such that:
\[
 \overline{L}=\overline{L}_1\otimes \overline{L}_2^{-1}.
\]

\subsection{Compacts operators}\label{paragrapheOpcompacts}
Consider two Hilbert spaces $\mathcal{H}$ and $\mathcal{H}'$. We denote by
$L(\mathcal{H},\mathcal{H}')$ the space of continuous
linear maps from $\h$ to $\h'$.

We say that $T$ in $L(\h,\h')$ is a compact operator if
and only if the image of every bounded subset of $\h$ by
$T$ is relatively compact in $\h'$. An operator $T$ is
said ton be finite rank operator if its image has a
finite dimension. In particular, it is a compact
operator. The dimension of its image is called the rank
of the operator.
We denote by $<,>$ the inner product of $\h$, and by
$\|\cdot\|$ the associated norm.
\begin{proposition}\label{operateurcompact}
Let $\mathcal{H}$ be a  Hilbert space, $\mathcal{B}(\mathcal{H})$ the space of bounded linear operators and
$\mathcal{K}(\mathcal{H})$ the space of compacts operators. We have, $\mathcal{K}(\mathcal{H})$ is a closed linear subspace of
$\mathcal{B}(\mathcal{H})$.
\end{proposition}
\begin{proof}
See for  example \cite[proposition 1.4]{functional2}.
\end{proof}

Let $T\in L(\h)$. For any $n\in \N$, we set
{{}
\begin{equation}\label{defvaleursinguliere}
\si_n(T)=\inf\Bigl\{ \bigl\|T-R\bigr\|:\, R\in L(\h), \text{rg}(R)\leq n\Bigr\},
\end{equation}
}
According to \cite[p. 232]{functional2}, $T$ is  compact if and only if the sequence $\bigl( \si_n(T)\bigr)_{n\in \N}$ converges
to $0$. We assume that $T$ is compact,
and  $\si_n(T)$ is called the n-th singular value of  $T$.

Let $P:=\bigl(T^\ast T\bigr)^{\frac{1}{2}}$ (where
$T^\ast$ denote adjoint operator of $T$), one shows that
$P$ is an positive selfadjoint compact operator; we
denote by $\bigl( \mu_n(T)\bigr)_{n\in \N}$ the set of  of nonzero  eigenvalues  $P$, in decreasing order and
counted with multiplicity (that is, each nonzero
eigenvalue  $\la$ appears $d_\la$ times, where $d_\la$
is the  dimension of $\ker(\la I-P)$).

We know that
\[
 \mu_n(T)=\si_n(T) \quad \forall\, n\in \N,
\]
see for  example \cite[p. 246]{functional2}.

\begin{definition}\label{definitionTrace}
\rm{Let $T$ be a compact operator. Let $\bigl(\mu_n(T)\bigr)_{n\in \N}$ the set of
singular values of $T$, in decreasing order. We
set
{{}
\[
 \|T\|_1:=\sum_{n\in \N} \mu_n(T).
\]
}
If $\|T\|_1<\infty$ then $T$ is called trace-class (or
nuclear) operator, and  $\|T\|_1$ denote its nuclear
norm.
}
\end{definition}

We denote by $\mathcal{C}_1(\h)$ the set of class trace
operators. We have:
\begin{proposition}
$\mathcal{C}_1(\h)$ is vectorial space, and $\|\cdot\|_1$ is a  norm on $\mathcal{C}_1(\h)$, called
the trace norm.
\end{proposition}
\begin{proof}
See \cite[15.11 problem 7, c]{Dieudonné2}.
\end{proof}

\begin{proposition}
Let $ T$ be a trace-class operator. Let
$\bigl(\xi_n\bigr)_{n\in \N}$ an orthonormal basis for
 $\h$. Then  {{} $\sum_{n\in \N}<T\xi_n,\xi_n>$} converges,  with sum equal to {{} $\|T\|_1$}.
\end{proposition}
\begin{proof}
See \cite[15.11 problem 7, b)]{Dieudonné2}.
\end{proof}
\begin{proposition}\label{lidskii}
Let $T$ be a trace-class operator, and be $(\la_n)_{n\in \N}$ the sequence of its eigenvalues counted with their
multiplicity.
Then, $\sum_{n\in \N}\la_n$ converges absolutely and we have
\[
\sum_{n\in \N}\la_n=\mathrm{Tr}(T).
\]
\end{proposition}
\begin{proof}
See \cite[(3.2)]{Simon}.
\end{proof}

\begin{proposition}\label{normetrace}
 Let $A$ et $B$ be two bounded operators and $T\in \mathcal{C}_1(\h)$, then
\[
 \|ATB\|_1\leq \|A\|\|T\|_1\|B\|.
\]
\end{proposition}

Let  $\mathcal{H}$ be a separable  Hilbert space. Let
$A$ be a compact operator on $\mathcal{H}$  and $\bigl(e_i\bigr)_i$ an  orthonormal basis of $\h$. If  $\sum_{i\geq 0}\bigl(Ae_i,e_i\bigr)$ is absolutely  convergent with respect to an  orthonormal basis $\bigl(e_i\bigr)_{i}$, and hence for any orthonormal
basis of $\mathcal{H}$, we call this sum the  trace of $A$ and we denote it by  $\mathrm{Tr}(A)$.

If $T$ is class trace operator, then
\[
 \|T\|_1=\mathrm{Tr}\bigl((T^\ast T)^\frac{1}{2}\bigr).
\]

Let $A$ be a class trace operator, we have:
\begin{itemize}
\item[$\bullet$] \begin{equation}\label{commutetrace11}
 \mathrm{Tr}(AB)=\mathrm{Tr}(BA),
\end{equation}

if   $B$ is bounded, cf. \cite[TR. 2 p.463]{functional}.
\item[$\bullet$]

\begin{equation}\label{tracenorme1}
 \bigl|\mathrm{Tr}(A) \bigr|\leq \|A\|_1.
\end{equation}

cf. \cite[TR. 7 p.463]{functional}.
\end{itemize}

\begin{proposition}
Let $\h$ be a Hilbert space. let $\bigl(<\cdot,
\cdot>_u\bigr)_{u\in I}$  a family of hermitian metrics
on  $\h$ uniformly equivalents. Let  $u_0\in I$, and $T\in \mathcal{C}_{1,u_0}(\h)$ then
 \begin{enumerate}
 \item  $T$ is a trace-class operator on $\h $ endowed
 with the norm $<\cdot,\cdot>_u$, for any $u\in I$,
\item  There exists $c_8$ et $c_9$ positives constants
such that:
\[
 c_8\|T\|_{1,u'}\leq \|T\|_{1,u}\leq c_9\|T\|_{1,u'} \quad \forall\, u,u'\in I.
\]
\end{enumerate}

\end{proposition}
\begin{proof}
Let us recall that  $L(\h)$ is endowed with the following
norm:

\[
 \|A\|=\sup_{x\in \h\setminus \{0\}}\frac{\|Ax\|}{\|x\|} \quad A\in L(\h).
\]
By assumption, there exist $c'_8,c'_9$ two positive
constants such that:
\[
 c'_8\|x\|_{u'}\leq \|x\|_u\leq  c'_9\|x\|_{u'}\quad \forall x\in \h\quad\forall\, u,u'\in I.
\]

Hence,
\[
\frac{c'_8}{c'_9}\frac{\|Tx\|_{u'}}{\|x\|_{u'}}\leq \frac{\|Tx\|_{u}}{\|x\|_{u}}\leq
\frac{c'_9}{c'_8}\frac{\|Tx\|_{u'}}{\|x\|_{u'}}\quad \forall\, x\neq 0.
\]

Therefore,
{{}
\[
 \frac{c'_8}{c'_9}\|T\|_{u'}\leq \|T\|_{u}\leq \frac{c'_9}{c'_8}\|T\|_{u'}.
\]
}
We deduce that $T$ is compact for any $u\in I$. Indeed;
 , If we consider$F$, a bounded closed subset in
$(\mathcal{H},<\cdot,\cdot>_{u})$ for some $u\in I$, then according to the last inequality, $F$ is bounded and
closed  for
$<,>_{u_0}$, and since $T\in \mathcal{C}_{1,u_0}(\mathcal{H})$, then (by definition)
$T$ is compact in $(\mathcal{H},<\cdot,\cdot>_{u_0})$.
It follows that $T(F)$ is  relatively  compact in the
previous space. Using the same inequality, we conclude that $T(F)$ is relatively compact in $(\mathcal{H},<,>_u)$.

Let $R$ be a  operator of finite rank, with rank less that$n$. We have

{{}
\[
 \frac{c'_8}{c'_9}\|T-R\|_{u'}\leq \|T-R\|_{u}\leq \frac{c'_9}{c'_8}\|T-R\|_{u'}.
\]
}
Then,
{{}
\[
 \frac{c'_8}{c'_9}\si_n(T)_{u'}\leq \si_n(T)_{u}\leq \frac{c'_9}{c'_8}\si_n(T)_{u'}.
\]
}
Therefore
\[
 \frac{c'_8}{c'_9}\|T\|_{1,u'}\leq\|T\|_{1,u}\leq \frac{c'_9}{c'_8}\|T\|_{1,u'}.
\]
\end{proof}

\begin{Corollaire}\label{equivtracenorme11}
 Let $(\vc_u)_{u\geq 1}$ a sequence of  hermitian metrics on $\h$, which   converges uniformly to a hermitian
 metric denoted $\vc_\infty$, when  $u$ goes to infinity $\vc_\infty$. We suppose that  $\forall \eps>0$, there
 exists $\eta>0$ such that
{{}
\begin{equation}\label{limitmulti}
(1-\eps)\|\xi\|_{u'}\leq  \|\xi\|_{u}\leq (1+\eps)\|\xi\|_{u'}\quad \forall u,u'>\eta.
\end{equation}
}
Then, we have for any  $0<\eps<1$

{{}
\[
\frac{1-\eps}{1+\eps}\si_n(T)_{u'}\leq  \si_n(T)_{u}\leq \frac{1+\eps}{1-\eps}\si_n(T)_{u'}\quad \forall
T\in L(\h),\,\forall n\in \N,\,\forall u,u'>\eta.
\]
}
In particular,

{{}
\[
\frac{1-\eps}{1+\eps}\|T\|_{1,u'}\leq  \|T\|_{1,u}\leq \frac{1+\eps}{1-\eps}\|T\|_{u'}\quad \forall T\in L(\h),\quad\forall u,u'>\eta.
\]
}

\end{Corollaire}
\begin{proof}
This is a consequence of the proof of the previous proposition.
\end{proof}

We consider the prehilbertian space $A^{0,0}(X)$, endowed with  the metric $\vc_{L^2,u}$ associated to the
metric  $h_{X,u}$. We
know that $(h_{X,u})_u$ converges uniformly to a  limit $h_{X,\infty}$. Hence  $\bigl(\vc_{L^2,u}\bigr)$ converges
uniformly to $\vc_{L^2,\infty}$, {{} (More precisely, it verifies  \eqref{limitmulti})}.

Let  $T=B e^{-t\Delta_{X,u}}$, where $B$ is a bounded operator and $t>0$ is  fixed. We have,  for any $u\gg 1$
{{}
\[
 \frac{1-\eps}{1+\eps}\si_n\bigl(Be^{-t\Delta_{X,u}}\bigr)_\infty\leq \si_n\bigl(Be^{-t\Delta_{X,u}}\bigr)_u\leq \frac{1+\eps}{1-\eps} \si_n\bigl(Be^{-t\Delta_{X,u}}\bigr)_\infty.
\]
}
Therefore,
{{}
\begin{equation}\label{thetainfty}
 \frac{1-\eps}{1+\eps}\bigl\|Be^{-t\Delta_{X,u}}\bigr\|_{1,\infty}\leq \bigl\|Be^{-t\Delta_{X,u}}\bigr\|_{1,u}\leq \frac{1+\eps}{1-\eps} \bigl\|Be^{-t\Delta_{X,u}}\bigr\|_{1,\infty}.
\end{equation}
}


We consider metric $L^2_\infty$ on $\h$. We have, for any $t>0$ fixed

\begin{equation}\label{thetainfty11}
 \begin{split}
  \bigl|\bigl\|Be^{-t\Delta_{X,u}}e^{-t\Delta_{X,\infty}} \bigr\|_{1,\infty}- \bigl\| Be^{-2t\Delta_{X,\infty}}\bigr\|_{1,\infty}\bigr|&\leq  \bigl\|Be^{-t\Delta_{X,u}}e^{-t\Delta_{X,\infty}}-  Be^{-2t\Delta_{X,\infty}}\bigr\|_{1,\infty}\\
&=\bigl\|B(e^{-t\Delta_{X,u}}-  e^{-t\Delta_{X,\infty}})e^{-t\Delta_{X,\infty}}\bigr\|_{1,\infty}\\
&\leq \bigl\|e^{-t\Delta_{X,u}}-  e^{-t\Delta_{X,\infty}}\bigr\|_{L^2,\infty}\bigl\|Be^{-t\Delta_{X,\infty}}\bigr\|_{1,\infty}
 \end{split}
\end{equation}

and
\begin{equation}\label{thetainfty12}
\begin{split}
 \Bigl|\bigl\|Be^{-t\Delta_{X,u}}e^{-t\Delta_{X,\infty}} \bigr\|_{1,\infty}- \bigl\| Be^{-2t\Delta_{X,u}}\bigr\|_{1,\infty}\Bigr|&\leq \bigl\|Be^{-t\Delta_{X,u}}e^{-t\Delta_{X,\infty}}- B e^{-2t\Delta_{X,u}}\bigr\|_{1,\infty}\\
&\leq \bigl\| Be^{-t\Delta_{X,u}}(e^{-t\Delta_{X,u}}-e^{-t\Delta_{X,\infty}})\bigr\|_{1,\infty}\\
&\leq \bigl\|e^{-t\Delta_{X,u}}-e^{-t\Delta_{X,\infty}}  \bigr\|_{L^2,\infty}\bigl\|Be^{-t\Delta_{X,u}}\|_{1,\infty}.
\end{split}
\end{equation}

\subsection{The heat kernel of Laplacians}
We follow \cite[Appendice, D]{Ma}, for the  definition of heat kernel  associated to a Laplacian. Let $\Delta$ be a
Laplacian, roughly speaking, one  defines heat kernel operator using the theory of operators, and it is denoted by
$e^{-t\Delta}$ . For any $t>0$, $e^{-t\Delta}$ is a compact operator on $L^2(X)$ to $L^2(X)$ which is $\mathcal{C}^1$ in $t$
and which verifies the following properties : for $s\in L^2(X)$
\[
\begin{split}
\Bigl( \frac{\pt}{\pt t}+\Delta \Bigr)e^{-t\Delta}s&=0,\\
\underset{t\mapsto 0}{\lim}\, e^{-t\Delta}s&=s\quad \text{dans}\quad L^2(X).
\end{split}
\]
One  shows that $e^{-t\Delta}$ is unique.

\begin{theorem}
Let  $\Delta^u$ be a  smooth family of  Laplacians, then for each $t>0$, the corresponding family
of heat kernels  $e^{-t\Delta_u}$ defines a smooth
family of operators on  $\mathcal{O}$. Furthermore,the derivative
of  $e^{-t\Delta_u}$ with respect to $u$ is given by
Duhamel formula:
{{}
\begin{equation}\label{heatkernel}
\frac{\partial }{\partial u}e^{-t\Delta^u}=-\int_{0}^t e^{-(t-s)\Delta^u}(\partial_u \Delta^u)e^{-s\Delta^u}ds.
\end{equation}
}
\end{theorem}

\begin{proof}
See \cite[theorem D.1.6]{Ma} ou \cite[theorem 2.48]{heat}.
\end{proof}

\begin{theorem}\label{semi}
Let $V$ be a  Banach space. If $A$ is   selfadjoint and
positive operator, then  $-A$ generates a  semi-group
$P(t)=e^{-tA}$ consisting of positive and selfadjoint operators
with norms $\leq 1$.
\end{theorem}
\begin{proof}
See \cite[Proposition 9.4]{Taylor}.
\end{proof}

\section{Appendix 2}\label{Quelqueslemmes}

In this section, we present some results and technical lemmas needed through this article.

\subsection{A technical result}
In the  following proposition, given a discrete sequence  $(h_n)_{n\in \N^\ast}$, of hermitian metrics on a
vector bundle on a Riemannian manifold, we construct a family  $(h_u)_{u\in \N^\ast}$  with continuous parameter
which
behaves smoothly and preserves some properties of the previous sequence. This construction, will be useful to
study the
infinitesimal variations of different objects attached to this sequence.

\begin{proposition}\label{suitefamille}
Let $X$ be a complex manifold. Let $E$ be a holomorphic vector bundle, for  example $TX$. We denote by $\mathcal{M}et(E)$
the space of continuous (integrable, smooth,...)  hermitian metrics on $E$. Let   $(h_n)_{n\in \N}$
be a sequence  hermitian metrics
on $E$ (non necessarily smooth), then there exists a continuous  family $(H(u))_{u\geq 1}$  such that:
\begin{enumerate}
\item  $H(u)$ is a hermitian  metric on  $E$, $\forall u$.
\item For any holomorphic local section $s$ of $E$, the following application
\[
 \begin{split}
 H: [1,\infty[&\lra \R^+\\
u&\longmapsto H(u)(s,s),
 \end{split}
\]
is smooth.
\item $H(n)=h_n$, $\forall\, n\in \N$.
 \item If we  suppose that $(h_n)_{n\in \N}$ converges uniformly to a metric $h_\infty$, then  the family
 $(\frac{H(u)}{h_\infty})_{u\geq 1}$ of smooth functions on $X$, converges uniformly to the constant function $1$ on $X$.
\item If $E$ is a line bundle,  then
{{}
\[\frac{\pt}{\pt u}  \log H(u)=\mathrm{O}\biggl(\frac{h_{[u]+1}-h_{[u]}}{h_{[u]}}  \biggr),\]}
on  $X$, with $[u]$ is the round up of  $u$.
\end{enumerate}

\end{proposition}

\begin{proof} Let  $(h_p)_{p\geq 2}$ be a discrete sequence of hermitian metrics. For any $n$, let $\rho_n$ be a smooth, positive
and nondecreasing   function on $\R^+$ such that:
\begin{equation}
\rho_n(x)=\left\{
    \begin{array}{ll}
        0,& \quad x\leq n   \\
        1, & \quad x\geq n+1
 \end{array}
\right.
\end{equation}

We can assume that $\rho_n(x)=\rho_1(x-n), \forall x\in \R^+ $.\\

We set $H_1:[1,\infty[\lra \mathcal{M}et$ such that $H_1(u)=h_1 $, $\forall u$. If $\h_2$ is the following function:
$\h_2(u)=(1-\rho_1(u))H_1(u)+\rho_1(u)h_2 $ then it is smooth and  verify $\h_2(1)=H_1(1)=h_1$ and $\h_2(2)=h_2$. By induction on
$k\in \N$, we set  $H_k(u)=(1-\rho_{k-1}(u))H_{k-1}(u)+\rho_{k-1}(u)h_k$, and we have  $H_k(i)=h_i, $ for $i\leq k-1$ and $H_k(k)=h_k$.

Then, we let $H:\R^+\lra \mathcal
{M}et(\mathcal{O})$ such that $H(u)=H_n(u)$ if $u\leq n-1$ (remark that   $H_{n+1}(u)=H_n(u)$) then $H$ is well defined, smooth and   $H(n)=h_n$ for any $n\in \N$.\\

Suppose that $(h_n)_{n\in \N}$ converges uniformly to  $h_\infty$. One proves, using induction on $k$, that {{}
\begin{equation}\label{1}
 H(u)=\bigl(1-\rho_{k-1}(u)\bigr)h_{k-1}+\rho_{k-1}(u)h_k\quad \forall\, u\in [k-1,k] \; \forall k\,\in \N^\ast.
\end{equation}}
Then, if $s$ a nonzero holomorphic local section of  $\mathcal{O}$ on an open subset $U$, we have:
\[
 \biggl|\frac{H(u)(s,s)}{h_\infty(s,s)}-1\biggr|\leq \biggl|\frac{h_{k-1}(s,s)}{h_\infty(s,s)}-1\biggr|+\biggl|\frac{h_k(s,s)}{h_\infty(s,s)}-1\biggr| \quad \forall\, u\in [k-1,k],
\]
then we get $(4.)$.  \\

We have, for any $u>1$:
\[
\begin{split}
\biggl|\frac{\pt }{\pt u}\log H(u)(s,s)\biggr|&=\bigl|h_u(s,s)^{-1}(\partial_u \rho_{k-1})(u)(h_k(s,s)-h_{k-1}(s,s))\bigr|\\
&=\bigl|\partial_u \rho_{k-1}(u) \bigr|\,\Bigl|\frac{h_k(s,s)-h_{k-1}(s,s)}{h_u(s,s)}\Bigr|\\
&\leq \bigl|\partial_u \rho_{k-1}(u) \bigr|\,\Bigl|\frac{h_k-h_{k-1}}{\min(h_{k-1},h_k)}\Bigr|\\
&=\bigl|\partial_u \rho_{k-1}(u) \bigr|\,\max\biggl(\frac{\bigl|h_k-h_{k-1}\bigr|}{h_{k-1}},\frac{\bigl|h_k-h_{k-1}\bigr|}{h_{k}} \biggr).
\end{split}
\]
Since  $\bigl|\partial_u \rho_{k-1}(u) \bigr|$  is uniformly bounded (take for instance, $\rho_k(x)=\rho_1(x-k)$, for any $k\in
\N_{\geq  2}$). Then, there exists a constant $M>0$ such that:
\[
 \biggl|\frac{\pt }{\pt u}\log H(u)\biggr|\leq M \frac{\bigl|h_{[u]+1}-h_{[u]}\bigr|}{h_{[u]}},
\]
for any $u\geq 1$.
\end{proof}



In this paper, we will denote by $h_u$, the metric $H(u)$.

\subsection{On the first nonzero eigenvalue of the Laplacian}
In this section, we recall a result due to Cheeger, which gives a geometric lower bound for the first nonzero eigenvalue of
the Laplacian   $\Delta$, such that $\overline{\mathcal{O}}_0$ is  the trivial line bundle endowed
with a constant metric. This lower bound can be expressed in  terms of the geometry of the manifold, see the theorem below.

 Let $(h_{X,p})_{p\in \N} $ be a bounded sequence of  smooth hermitian metrics on $X$, if we denote by
 $\la_{p,1}$,  for any $p\in \N$, the first nonzero eigenvalue of  the Laplacian attached to
 $\bigl((TX,h_{X,p}),\overline{\mathcal{O}}_0 \bigr)$, then will show there exists a constant $\kappa>0$ such that
\[
 \la_{p,1}\geq \kappa,\quad \forall\, p\in\N.
\]
\begin{definition}[Cheeger isopermetric constant]\label{Cheeger0} \rm{Let  $(M,g)$ be a compact riemannian manifold of dimension $n$
without boundary. We set:
\[
  h_g(M):=\inf\frac{A(S)}{\min\bigl(V(M_1),V(M_2 )\bigr)},
 \]
where $A(\cdot)$ denote the $(n-1)$ dimensional volume, $V(\cdot)$ is the  volume and the inf is taken over the
set of  compact submanifolds with corners   $S$ of  dimension $n-1$, $M_1$ and $M_2$ are the two open submanifolds
with boundary such
that $M=M_1\cup M_2$ and $\partial M_j=S$, for $j=1,2$.}
\end{definition}

\begin{theorem}\label{Cheeger}
Let $\lambda_1$ be the first nonzero eigenvalue of the Laplacian associated to $(M,g)$, then
\[
 \lambda_1\geq \frac{1}{4}h^2_g(M).
\]

\end{theorem}
\begin{proof}
 See \cite{Cheeger}.
\end{proof}
\begin{remarque}
\rm{
\begin{enumerate}
 \item $h_{tg}(M)=t^{-\frac{1}{2}}h_g(M)$, for any $t>0$. (follows from the definition of $h_g(M)$).
 \item $h_g(M)$ is a positive real, see \cite[p.198]{Cheeger} for  $n=2$.
\end{enumerate}
}
\end{remarque}

\begin{proposition}\label{Cheeger3}
Let us consider a bounded sequence $\bigl(g_{p} \bigr)_{p\in \N_{\geq 2}}$ of smooth riemannian metrics on $X$, namely, we suppose
there exist $c_1$ and $c_2$ two nonzero constants such that:
\begin{equation}\label{gpq}
c_1\, g_{p} \leq g_{q}\leq c_2 \, g_{p}\quad \forall \, p,q\in \N.
\end{equation}
 Then
\begin{equation}\label{infboundvp}
\frac{c_1}{c_2} h_{g_p}(M)\leq h_{g_q}(M)\leq \frac{c_2}{c_1}h_{g_p}(M)\quad \forall\, p,q\in \N.
\end{equation}

\end{proposition}
\begin{proof}

It suffices to note that \eqref{gpq} is stable by restriction to submanifolds of $M$.
\end{proof}

\begin{proposition}\label{lowerbound}
Let $(g_p)_{p\in \N}$ be as before, then there exists $\kappa$ a positive constant such that:
\[
 \lambda_{p,1}\geq \kappa\quad   \forall p\in \N,
\]
where $\lambda_{p,1}$ is the first nonzero eigenvalue of the  Laplacian associated to $g_p$.
\end{proposition}
\begin{proof}
It follows from \eqref{Cheeger} and \eqref{Cheeger3}.
\end{proof}

\bibliographystyle{plain}
\bibliography{biblio}

\begin{thebibliography}{10}

\bibitem{Table2}
Milton Abramowitz and Irene~A. Stegun, editors.
\newblock {\em Handbook of mathematical functions with formulas, graphs, and
  mathematical tables}.
\newblock Dover Publications Inc., New York, 1992.
\newblock Reprint of the 1972 edition.

\bibitem{heat}
Nicole Berline, Ezra Getzler, and Mich{\`e}le Vergne.
\newblock {\em Heat kernels and {D}irac operators}.
\newblock Grundlehren Text Editions. Springer-Verlag, Berlin, 2004.
\newblock Corrected reprint of the 1992 original.

\bibitem{BGS1}
J.-M. Bismut, H.~Gillet, and C.~Soul{\'e}.
\newblock Analytic torsion and holomorphic determinant bundles. {I}.
  {B}ott-{C}hern forms and analytic torsion.
\newblock {\em Comm. Math. Phys.}, 115(1):49--78, 1988.

\bibitem{Buser}
Peter Buser.
\newblock {\em Geometry and spectra of compact {R}iemann surfaces}.
\newblock Modern Birkh\"auser Classics. Birkh\"auser Boston Inc., Boston, MA,
  2010.
\newblock Reprint of the 1992 edition.

\bibitem{Cheeger}
Jeff Cheeger.
\newblock A lower bound for the smallest eigenvalue of the {L}aplacian.
\newblock In {\em Problems in analysis ({P}apers dedicated to {S}alomon
  {B}ochner, 1969)}, pages 195--199. Princeton Univ. Press, Princeton, N. J.,
  1970.

\bibitem{Dieudonné2}
J.~Dieudonn{\'e}.
\newblock {\em \'{E}l\'ements d'analyse. {T}ome {II}: {C}hapitres {XII} \`a
  {XV}}.
\newblock Cahiers Scientifiques, Fasc. XXXI. Gauthier-Villars, \'Editeur,
  Paris, 1968.

\bibitem{Character}
Henri Gillet and Christophe Soul{\'e}.
\newblock Characteristic classes for algebraic vector bundles with {H}ermitian
  metric. {I}.
\newblock {\em Ann. of Math. (2)}, 131(1):163--203, 1990.

\bibitem{character2}
Henri Gillet and Christophe Soul{\'e}.
\newblock Characteristic classes for algebraic vector bundles with {H}ermitian
  metric. {II}.
\newblock {\em Ann. of Math. (2)}, 131(2):205--238, 1990.

\bibitem{GH}
Phillip Griffiths and Joseph Harris.
\newblock {\em Principles of algebraic geometry}.
\newblock Wiley Classics Library. John Wiley \& Sons Inc., New York, 1994.
\newblock Reprint of the 1978 original.

\bibitem{Mounir1}
Mounir Hajli.
\newblock {S}pectre du {L}aplacien singulier associé aux métriques canoniques
  sur la droite projective complexe.
\newblock {\em Arxiv}.

\bibitem{Mounir2}
Mounir Hajli.
\newblock {S}ur la fonction zêta associée au laplacien singulier associé aux
  métriques canoniques sur la droite projective complexe.
\newblock {\em Arxiv}.

\bibitem{these}
Mounir Hajli.
\newblock \emph{{T}héorie spectrale pour certaines métriques singulières et
  {G}éométrie d'{A}rakelov}.
\newblock {\em \rm{PhD thesis}}, Université Pierre et Marie Curie, 2012.

\bibitem{functional2}
Francis Hirsch and Gilles Lacombe.
\newblock {\em Elements of functional analysis}, volume 192 of {\em Graduate
  Texts in Mathematics}.
\newblock Springer-Verlag, New York, 1999.
\newblock Translated from the 1997 French original by Silvio Levy.

\bibitem{compactriemannsurfaces}
J{\"u}rgen Jost.
\newblock {\em Compact {R}iemann surfaces}.
\newblock Universitext. Springer-Verlag, Berlin, third edition, 2006.
\newblock An introduction to contemporary mathematics.

\bibitem{functional}
Serge Lang.
\newblock {\em Real and functional analysis}, volume 142 of {\em Graduate Texts
  in Mathematics}.
\newblock Springer-Verlag, New York, third edition, 1993.

\bibitem{Ma}
Xiaonan Ma and George Marinescu.
\newblock {\em Holomorphic {M}orse inequalities and {B}ergman kernels}, volume
  254 of {\em Progress in Mathematics}.
\newblock Birkh\"auser Verlag, Basel, 2007.

\bibitem{RaySinger}
D.~B. Ray and I.~M. Singer.
\newblock Analytic torsion for complex manifolds.
\newblock {\em Ann. of Math. (2)}, 98:154--177, 1973.

\bibitem{Simon}
Barry Simon.
\newblock {\em Trace ideals and their applications}, volume 120 of {\em
  Mathematical Surveys and Monographs}.
\newblock American Mathematical Society, Providence, RI, second edition, 2005.

\bibitem{Soulé}
C.~Soul{\'e}.
\newblock {\em Lectures on {A}rakelov geometry}, volume~33 of {\em Cambridge
  Studies in Advanced Mathematics}.
\newblock Cambridge University Press, Cambridge, 1992.
\newblock With the collaboration of D. Abramovich, J.-F. Burnol and J. Kramer.

\bibitem{Taylor}
Michael~E. Taylor.
\newblock {\em Partial differential equations {I}. {B}asic theory}, volume 115
  of {\em Applied Mathematical Sciences}.
\newblock Springer, New York, second edition, 2011.

\bibitem{Voisin}
Claire Voisin.
\newblock {\em Th\'eorie de {H}odge et g\'eom\'etrie alg\'ebrique complexe},
  volume~10 of {\em Cours Sp\'ecialis\'es [Specialized Courses]}.
\newblock Soci\'et\'e Math\'ematique de France, Paris, 2002.

\end{thebibliography}

\vspace{1cm}

\begin{center}
{\sffamily \noindent National Center for Theoretical Sciences, (Taipei Office)\\
 National Taiwan University, Taipei 106, Taiwan}\\

 {e-mail}: {hajli@math.jussieu.fr}

\end{center}

\end{document}